\documentclass[11pt]{article}
\usepackage[T1]{fontenc}
\usepackage{lmodern}
\usepackage[a4paper,margin=1in]{geometry}
\usepackage{amsmath,amssymb,amsthm,mathtools,amsfonts}
\usepackage{enumitem}
\usepackage{microtype}
\usepackage{graphicx}
\usepackage{subcaption}
\usepackage{tikz-cd}
\usepackage[colorlinks=true,linkcolor=blue,citecolor=blue,urlcolor=blue]{hyperref}
\usepackage{tikz}
\usetikzlibrary{calc}
\newtheorem{theorem}{Theorem}[section]
\newtheorem{lemma}[theorem]{Lemma}
\newtheorem{proposition}[theorem]{Proposition}
\newtheorem{corollary}[theorem]{Corollary}

\theoremstyle{definition}
\newtheorem{definition}[theorem]{Definition}
\newtheorem{remark}[theorem]{Remark}
\newtheorem{example}[theorem]{Example} 

\newcommand{\hor}{\mathrm{hor}}
\newcommand{\diam}{\mathrm{diam}}
\newcommand{\CD}{\mathrm{CD}}
\newcommand{\Del}{\Delta}
\newcommand{\R}{\mathbb R}

\newcommand{\Spec}{\operatorname{Spec}}

\newcommand{\one}{\mathbf 1}

\newcommand{\card}[1]{\left|#1\right|}

\newcommand{\vis}{\operatorname{vis}}
\newcommand{\col}{\operatorname{col}}

\newcommand{\Ric}{\operatorname{Ric}}
\title{Regular Lichnerowicz-sharp graphs are hypercube bundles}
\author{%
Yanlong Ding\thanks{School of Mathematical Sciences, University of Science and Technology of China, Hefei 230026, China. Email address: \texttt{dylustc@mail.ustc.edu.cn}.}%
\and
Shiping Liu\thanks{School of Mathematical Sciences, University of Science and Technology of China, Hefei 230026, China. Email address: \texttt{spliu@ustc.edu.cn}.}%
\and
Chiyu Zhou\thanks{School of Mathematical Sciences, University of Science and Technology of China, Hefei 230026, China. Email address: \texttt{dovong@mail.ustc.edu.cn}.}%
}
\date{\today}

\begin{document}
\maketitle

\begin{abstract}
Hypercube graphs are fundamental model spaces of positive curvature in discrete comparison geometry. Let $G$ be a finite, connected, simple, unweighted graph with Bakry--\'Emery curvature bounded below by \(K\). 
We call $G$ Lichnerowicz-sharp if its first non-zero non-normalized Laplacian eigenvalue $\lambda_1=K$. We prove that all regular Lichnerowicz-sharp graphs are hypercube bundles with constant Bakry-\'Emery curvature $2$. A hypercube bundle is a graph bundle whose fiber graphs are hypercubes. As applications, we show that any $d$-regular Lichnerowicz-sharp graph with the multiplicity $m_K(G)$ of $\lambda_1=K$ at least $d/2$ can split off a hypercube of certain dimension. Moreover, we characterize all $d$-regular Lichnerowicz-sharp graphs with $m_{K}(G)\geq d-3$. 

Interestingly, our result leads to the following spectral rigidity theorem of hypercubes. For a graph $G$ with maximum degree $\Delta$, if the multiplicity $m_K(G)\geq \Delta-1$, then $G$ is a $\Delta$-dimensional hypercube. This improves, in the unweighted setting, the multiplicity condition
\(m_K(G)\geq \Delta\) appearing in the hypercube rigidity theorem of Liu,
M\"unch, and Peyerimhoff. This improvement is optimal. 

\end{abstract}

\section{Introduction}
Hypercube graphs are fundamental discrete structures that have been widely studied in geometry \cite{Gromov1999}, probability theory \cite{Bobkov1997}, graph theory \cite{Harary1988}, and other related fields. In this paper, motivated by the role of hypercubes as model spaces in discrete comparison geometry, we investigate the equality cases in the Bakry--\'Emery Lichnerowicz eigenvalue estimate for graphs and ask
what global structures are enforced by such sharpness.

\subsection{Continuous and discrete spectral comparison and rigidity}
The Lichnerowicz theorem \cite{Lichnerowicz1958} is a classical result in
comparison geometry. It asserts that if a \(d\)-dimensional closed Riemannian
manifold \(M\) has Ricci curvature bounded below by that of the round sphere
\(\mathbb{S}^d\), then its first non-zero Laplace--Beltrami eigenvalue satisfies
\(
        \lambda_1(M)\geq \lambda_1(\mathbb{S}^d).
\)
Obata \cite{Obata1962} proved the corresponding rigidity statement: equality
holds in the Lichnerowicz estimate if and only if \(M\) is isometric to
\(\mathbb{S}^d\). 
Related rigidity and pinching results were further developed by Cheng,
Petersen, and Aubry \cite{Cheng1975,Petersen1999,Aubry2005}.  

In this paper, we study discrete analogues of Obata's rigidity theorem within the framework of Bakry--\'Emery curvature. The \(\Gamma\)-calculus and curvature dimension conditions of Bakry and \'Emery \cite{BE1985} provide a
powerful synthetic notion of Ricci curvature, which has been developed extensively; see, for example, \cite{BGL2014}. In the graph setting, discrete Bakry--\'Emery theory was initiated in \cite{elworthy,michael,riccicurlinyau}
and has since become an effective tool for studying geometric, analytic, and combinatorial aspects of graphs, see,
for instance, \cite{BHLLMY2015,FathiShu2018,HornJGT,Hua2019,HM2024, MM2024, KMY2021,LMP2018,MuenchRose2020,SalezGAFA,SalezJEMS} and the references therein.

Unless otherwise stated, all graphs considered in this paper are simple, finite, connected, and unweighted. For a graph $G=(V,E)$ with vertex set $V$ and edge set $E$, we list the eigenvalues of its non-normalized Laplacian as below:
\[0=\lambda_0<\lambda_1\leq \cdots\leq\lambda_{|V|-1}.\]
If $G$ satisfies the Bakry--\'Emery curvature dimension condition $\CD(K,\infty)$ for some $K>0$, then we have the following Lichnerowicz type eigenvalue estimate (see e.g. \cite{curvatureaspectsofgraphs,curvatureandhigherorder}),
\begin{equation}\label{eq:Lich}\lambda_1\geq K.\end{equation}
For the definition of the Bakry--\'Emery curvature dimension condition $\CD(K,\infty)$, we refer to Definition \ref{def:CD} below.
Let $K_2$ be the complete graph with two vertices. The $d$-dimensional hypercube graph $H_d$ is isomorphic to the Cartesian product of $d$ copies of $K_2$. Notice that $H_d$ is $d$-regular and its eigenvalues satisfy 
\[
        2=\lambda_1(H_d)=\cdots=\lambda_d(H_d)<\lambda_{d+1}(H_d).
\]
 Moreover, $H_d$ satisfies $\CD(2,\infty)$ (see \cite[Example 7.15]{CLP2020}). That is, for the hypercube graph $H_d$, the equality in \eqref{eq:Lich} holds.
Under a high-multiplicity assumption, Liu, M\"unch, and Peyerimhoff \cite{LMP} proved the following discrete Obata-type rigidity theorem.
\begin{theorem}[{\cite[Theorem~1.4]{LMP}}]\label{lem:LMP-rigidity}
Suppose that $G$ satisfies $\CD(K,\infty)$ with $K>0$ and
\begin{equation}\label{eq:LMP}
     \lambda_{\Delta}=K,
\end{equation}
where $\Del=\max_x\deg(x)$ is the maximum degree of $G$. Then $G\cong H_\Del$ and, in particular, $K=2$.
\end{theorem}

In this paper, we study what happens under weaker assumptions, such as 
\(\lambda_{\Delta-1}=K\), \(\lambda_{\Delta-2}=K\), or simply \(\lambda_1=K\),
for graphs satisfying \(\CD(K,\infty)\).

\subsection{Main results}
We call a graph \emph{Lichnerowicz-sharp} if it satisfies
\(\CD(K,\infty)\) and \(\lambda_1=K\).

Unlike in the continuous setting, where sharpness can occur for any
\(K>0\), the unweighted graph setting forces the curvature parameter to be
uniquely determined. We first prove that the conclusion \(K = 2\) in Theorem \ref{lem:LMP-rigidity} already follows from the weaker
assumption \(\lambda_1 = K\). To make this precise, let
\[
\mathcal{K}_{\infty}(x)
=
\sup\{K:\ G \text{ satisfies } \CD(K,\infty)\text{ at }x\}
\]
denote the \(\infty\)-Bakry--Émery curvature at \(x\).
Theorem~\ref{thm:first-sharp-K2} shows that every Lichnerowicz-sharp graph
satisfies
\[
\mathcal{K}_{\infty}\equiv2.
\]

In contrast to Obata's rigidity in the continuous setting, where the round sphere is the unique equality case, Lichnerowicz sharpness in the graph setting admits a much richer class of examples. Indeed, if a graph \(G\) is Lichnerowicz-sharp and another graph \(B\) satisfies \(\CD(2,\infty)\), then their Cartesian product \(G\square B\) is again Lichnerowicz-sharp.

However, not all Lichnerowicz-sharp graphs are decomposable as nontrivial Cartesian products. Indeed, Examples~\ref{ex:noncartesian-k4-h2-cd2-lambda2} and
\ref{ex:H-2n-plus-5-quotient} provide Lichnerowicz-sharp graphs that admit no
nontrivial Cartesian product decomposition.
\begin{figure}[htpb]
  \centering
  \begin{subfigure}[b]{0.47\textwidth}
    \centering
    \resizebox{\linewidth}{!}{\begin{tikzpicture}[
  scale=1.35,
  line cap=round,
  line join=round,
  bundle/.style={black, line width=0.72pt},
  hidden/.style={black, line width=0.48pt,
    dash pattern=on 4.2pt off 3.2pt},
  fiber/.style={line width=1.25pt},
  every node/.style={font=\large}
]

\coordinate (f1-00) at (0.00, 0.00);
\coordinate (f0-00) at (1.64, 1.96);
\coordinate (f2-00) at (4.52,-0.45);
\coordinate (f3-00) at (6.18, 1.55);

\foreach \i in {0,1,2,3} {
  \coordinate (f\i-01) at ($(f\i-00)+(0.00,1.20)$);
  \coordinate (f\i-10) at ($(f\i-00)+(0.95,0.53)$);
  \coordinate (f\i-11) at ($(f\i-00)+(0.95,1.73)$);
}

\foreach \i/\j in {0/1,0/2,0/3,1/2,1/3,2/3} {
  \foreach \p in {01,10,11} {
    \draw[hidden] (f\i-\p) -- (f\j-\p);
  }
}

\foreach \i/\j in {0/1,0/2,0/3,1/2,1/3,2/3} {
  \draw[bundle] (f\i-00) -- (f\j-00);
}

\draw[fiber,red!85!black]
  (f1-00) -- (f1-01) -- (f1-11) -- (f1-10) -- cycle;

\draw[fiber,green!55!black]
  (f0-00) -- (f0-01) -- (f0-11) -- (f0-10) -- cycle;

\draw[fiber,blue!80!black]
  (f2-00) -- (f2-01) -- (f2-11) -- (f2-10) -- cycle;

\draw[fiber,orange!90!black]
  (f3-00) -- (f3-01) -- (f3-11) -- (f3-10) -- cycle;

\node[red!85!black,anchor=south east]
  at ($(f1-01)+(-0.12,0.02)$) {$V_1$};
\node[green!55!black,anchor=south east]
  at ($(f0-11)+(0.16,-0.02)$) {$V_0$};
\node[blue!80!black,anchor=north west]
  at ($(f2-11)+(0.12,-0.10)$) {$V_2$};
\node[orange!90!black,anchor=south east]
  at ($(f3-11)+(0.20,-0.02)$) {$V_3$};

\end{tikzpicture}}
    \caption{Examples~\ref{ex:noncartesian-k4-h2-cd2-lambda2}}
    \label{fig:twisted-K4-H2-bundle}
  \end{subfigure}
  \hfill
  \begin{subfigure}[b]{0.47\textwidth}
    \centering
    \resizebox{0.66\linewidth}{!}{\begin{tikzpicture}[
  x={(1cm,0cm)},
  y={(0cm,1cm)},
  z={(0.56cm,0.34cm)},
  vertex/.style={circle, fill=red, inner sep=1.25pt},
  fiberedge/.style={red, line width=0.45pt, line cap=round},
  bundleedge/.style={black, line width=0.55pt,
    dash pattern=on 5pt off 3pt, line cap=round}
]


\def\fx{0.32}
\def\fy{0.27}
\def\innerR{2.20}
\def\outerR{3.95}

\foreach \base/\rad/\ang in {0000/\innerR/67.5,0001/\innerR/22.5,0011/\innerR/-22.5,0010/\innerR/-67.5,0110/\innerR/-112.5,0111/\innerR/-157.5,0101/\innerR/157.5,0100/\innerR/112.5,1000/\outerR/90,1001/\outerR/45,1011/\outerR/0,1010/\outerR/-45,1110/\outerR/-90,1111/\outerR/-135,1101/\outerR/180,1100/\outerR/135} {
  \coordinate (F-\base) at ({\rad*cos(\ang)},{\rad*sin(\ang)},0);
  \coordinate (v-\base-00) at ({\rad*cos(\ang)+\fx*sin(\ang)-\fy*cos(\ang)},{\rad*sin(\ang)-\fx*cos(\ang)-\fy*sin(\ang)},0);
  \coordinate (v-\base-10) at ({\rad*cos(\ang)-\fx*sin(\ang)-\fy*cos(\ang)},{\rad*sin(\ang)+\fx*cos(\ang)-\fy*sin(\ang)},0);
  \coordinate (v-\base-01) at ({\rad*cos(\ang)+\fx*sin(\ang)+\fy*cos(\ang)},{\rad*sin(\ang)-\fx*cos(\ang)+\fy*sin(\ang)},0);
  \coordinate (v-\base-11) at ({\rad*cos(\ang)-\fx*sin(\ang)+\fy*cos(\ang)},{\rad*sin(\ang)+\fx*cos(\ang)+\fy*sin(\ang)},0);
}

\foreach \a/\b in {0000/1000,0001/1001,0010/1010,0011/1011,0100/1100,0101/1101,0110/1110,0111/1111,0000/0100,0001/0101,0010/0110,0011/0111,1000/1100,1001/1101,1010/1110,1011/1111,0000/0010,0001/0011,0100/0110,0101/0111,1000/1010,1001/1011,1100/1110,1101/1111,0000/0001,0010/0011,0100/0101,0110/0111,1000/1001,1010/1011,1100/1101,1110/1111} {
  \foreach \p in {00,10,01,11} {
    \draw[bundleedge] (v-\a-\p) -- (v-\b-\p);
  }
}

\foreach \a/\b in {0000/1111,0001/1110,0010/1101,0011/1100,0100/1011,0101/1010,0110/1001,0111/1000} {
  \draw[bundleedge] (v-\a-00) -- (v-\b-00);
  \draw[bundleedge] (v-\a-10) -- (v-\b-01);
  \draw[bundleedge] (v-\a-01) -- (v-\b-10);
  \draw[bundleedge] (v-\a-11) -- (v-\b-11);
}

\foreach \base in {0000,0001,0010,0011,0100,0101,0110,0111,1000,1001,1010,1011,1100,1101,1110,1111} {
  \draw[fiberedge] (v-\base-00) -- (v-\base-10);
  \draw[fiberedge] (v-\base-00) -- (v-\base-01);
  \draw[fiberedge] (v-\base-10) -- (v-\base-11);
  \draw[fiberedge] (v-\base-01) -- (v-\base-11);
}

\foreach \base in {0000,0001,0010,0011,0100,0101,0110,0111,1000,1001,1010,1011,1100,1101,1110,1111} {
  \foreach \p in {00,10,01,11} {
    \node[vertex] at (v-\base-\p) {};
  }
}

\end{tikzpicture}}
    \caption{Example~\ref{ex:H-2n-plus-5-quotient}}
    \label{fig:G1-noncartesian}
  \end{subfigure}
  \label{fig:noncartesian-examples}
\end{figure}

Nevertheless, these examples still exhibit a remarkable fibred structure.

Recall that a classical fibre bundle is locally a product of a base space and a fixed fibre, with transition maps gluing these local product structures. A graph bundle is a discrete analogue of this construction. It was introduced as a common generalization of graph coverings and Cartesian products, and has since been extensively studied from structural, topological, and algorithmic perspectives; see, for example, \cite{PTSJVJ1983,MPTSM1988,KJL1990,SohnLee1994,ChaeKwakLee1993,ImrichPisanskiZerovnik1997,HongKwakLee1999,KlavzarMohar1995,KwakLeeSohn1996,ZmazekZerovnik2000,ZmazekZerovnik2002,ZmazekZerovnik2002a,ZmazekZerovnik2006,Zerovnik2000,BanicErvesZerovnik2009,BanicZerovnik2010,ErvesZerovnik2013,FengKwak2006,KwakKwon2001,PisanskiZerovnik2009}; see also \cite{LL2024} for recent developments.

Our main structural result shows that graph bundles provide precisely the right framework for regular Lichnerowicz-sharp graphs: every regular Lichnerowicz-sharp graph is a graph bundle whose fibre is a hypercube. Thus, although global Cartesian product decompositions fail in general, a strong fibred rigidity persists.

We now briefly recall the definition of a graph bundle.

\begin{definition}[{\cite[Definiation 1]{LL2024}}]
    Let $G$ and $F$ be two unweighted graphs. Let $E^{\mathrm{ori}}_{G}=\{(x,y): x\sim y \text{ in } G\}$ be the set of oriented edges of $G$. Let $\sigma: E^{\mathrm{ori}}_G\to \operatorname{Aut}(F), (x,y)\to \sigma_{xy}$ such that $\sigma_{yx}=\sigma_{xy}^{-1}$. Construct the graph $G\square_\sigma F$ by setting $V(G\square_\sigma F)=V(G)\times V(F)$ and $E(G\square_\sigma F)$
    \begin{enumerate}[label=\textup{(\roman*)}]
        \item $(x,u)\sim (y,\sigma_{xy}(u))$ if $x\sim y$ in $G$;
        \item $(x,u)\sim (x,v)$ if $u\sim v$ in $F$.
    \end{enumerate}
    We call $G\square_\sigma F$ the $F$-bundle over $G$ and refer to $G$ and $F$ as the base graph and the fiber graph, respectively.
\end{definition}

Through the notation of graph bundle, we can state our main theorem.

\begin{theorem}
  \label{tm:Bundle-structure1}
  Let $G$ be a connected \(d\)-regular Lichnerowicz-sharp graph. Let $m_2(G)$ denote the multiplicity of the $2$-eigenvalue. Then there exists a graph \(G^\prime\) and an integer \(r\geq m_2(G)\) such that \(G\) is an \(H_r\)-bundle over \(G^\prime\). The base graph \(G^\prime\) satisfies \(\CD(2,\infty)\), and if it is nontrivial, then \(\lambda_1(G^\prime)>2\). 
\end{theorem}
In particular, if the base graph \(G^\prime\) is trivial, then \(G\cong H_r\).

We next study the converse problem: when is a hypercube bundle
\(G^\prime\square_\sigma H_r\) Lichnerowicz-sharp?  Cartesian products provide
the basic examples.  If \(G^\prime\) satisfies \(\CD(2,\infty)\) and
\(\lambda_1(G^\prime)>2\), then \(G^\prime\square H_r\)  is Lichnerowicz
sharp.  For general bundles over a base with \(\lambda_1(G^\prime)>2\), the
\(2\)-eigenspace is controlled by how the transition maps act on the
\(2\)-eigenspace of \(H_r\).  This leads to splitting and triviality criteria:
large families of \(\{-1,1\}\)-valued \(2\)-eigenfunctions detect product
hypercube factors, and the dimension of $2$-eigenspace equal to $r$ is equivalent to the bundle being \(H_r\)-bundle isomorphic to the trivial
product \(G^\prime\square H_r\).

We also analyze curvature on hypercube bundles.  Under the section condition
arising from the structural theorem, the curvature matrix decomposes into a
fiber part and a horizontal part.  This direct-sum decomposition reduces
several curvature questions for bundles to curvature questions on the base and
on the hypercube fiber.  As an application, we characterize locally cubical
graphs by the identity \(A_\infty(x)=2I\) at every vertex, and for locally
cubical base graphs determine when a hypercube bundle satisfies
\(\CD(2,\infty)\).  The converse analysis also produces nontrivial examples:
for every \(n\geq4\) and \(r\geq2\), there are nontrivial Lichnerowicz-sharp
\(H_r\)-bundles over \(K_n\), whereas every Lichnerowicz-sharp \(H_r\)-bundle
over \(K_3\) is trivial.

Finally, we derive rigidity and splitting consequences from the multiplicity
\(m_2(G)\). Assume $G$ is a $H_r$-bundle over a base graph $G'$.  The observation $m_2(G)\leq r$ yields the following rigidity results.
\begin{theorem}
  \label{tm:Bundle-structure2}
  Let \(G\) be a connected \(d\)-regular Lichnerowicz-sharp graph. We have the following rigidity consequences:
  \begin{enumerate}[label=\textup{(\roman*)}]
  \item\label{item1: d-1} If \(m_2(G)\geq d-1\), then \(G\cong H_d\).
  \item\label{item2: d-2} If \(m_2(G)=d-2\), then \(G\cong H_{d-2}\square K_3\).
  \item\label{item3: d-3} If \(m_2(G)=d-3\), then
    \[
      G\cong H_{d-3}\square K_4
      \quad\text{or}\quad
      G\cong H_{d-3}\square K_{3,3}.
    \]
  \end{enumerate}
\end{theorem}
These results show how the bundle theorem connects hypercube rigidity with
product examples involving lower-dimensional hypercube factors.

We also obtain an optimal spectral rigidity theorem without assuming
regularity in advance.  If \(G\) is connected, satisfies \(\CD(2,\infty)\),
and \(\Delta=\max_x\deg(x)\), then
\[
  \lambda_{\Delta-1}=2
  \quad\Longrightarrow\quad
  G\cong H_\Delta .
\]
The key additional ingredient is a local regularity theorem: under the
appropriate sharp spectral hypothesis, the curvature condition forces all
degrees to be equal.  

The paper ends with a general splitting result:
\begin{theorem}
  \label{thm:Lichnerowicz-sharpK_2Factor}
    Let $G$ be a connected $d$-regular Lichnerowicz-sharp graph. Assume $m_2(G)\geq d-m$ for some positive integer $m\leq \frac{d}{2}$. Then $G\cong B\square H_{d-2m+1}$ for some connected $(2m-1)$-regular graph $B$. Moreover, $B$ is also $\CD(2,\infty)$ and $m_2(B)\geq m-1$.
\end{theorem}

\subsection{Continuous splitting versus graph-bundle rigidity}
The Bakry--\'Emery curvature theory of graphs is, in some respects, more closely analogous to that of weighted Riemannian manifolds, see \cite[Section 7]{CKLP} and \cite[Section 1.6]{CLP2020}. In the weighted Riemannian setting, equality in the Lichnerowicz estimate admits a natural splitting interpretation. Cheng and Zhou made the role of the multiplicity of the sharp eigenvalue explicit in the following theorem.
\begin{theorem}[{\cite[Theorem~2]{CZ2017}}]
Let \((M^n,g,e^{-f}dV_g)\) be a complete smooth weighted Riemannian manifold
with \(\Ric_f\geq Kg\) for some positive constant \(K\). Then
\(\lambda_1(\Delta_f)=K\) with multiplicity \(k\geq1\) if and only if
\begin{equation*}
(M^n,g,e^{-f}dV_g)
\cong
(\Sigma^{n-k},g_\Sigma,e^{-f\vert_{\Sigma\times\{0\}}}dV_{g_\Sigma})
\times
(\mathbb{R}^k,g_{\mathrm{Eucl}},e^{-K\frac{\vert x\vert^2}{2}}dx),
\end{equation*}
where \(\Sigma^{n-k}\) is an $(n-k)$ dimensional submanifold satisfying
\(\Ric_f^{\Sigma}\geq K\Delta_f^{\Sigma}\) and
\(\lambda_1(\Delta_f^{\Sigma})>K\).
\end{theorem}

As a consequence, they obtained the following rigidity result when the sharp
eigenvalue has multiplicity at least \(n-1\).
\begin{corollary}[{\cite[Corollary~2]{CZ2017}}]
Let \((M^n,g,e^{-f}dV_g)\) be a complete smooth weighted Riemannian manifold
with \(\Ric_f\geq Kg\) for some positive constant \(K\). If the multiplicity
of \(\lambda_1(\Delta_f)=K\) is at least \(n-1\), then \((M,g)\) is isometric
to the Euclidean space \((\mathbb{R}^n,g_{\mathrm{Eucl}})\), and \(f\) can be
expressed as
\begin{equation*}
f(x_1,\ldots,x_n)
=
\varphi(x_1)+\frac{K(x_2^2+\cdots+x_n^2)}{2},
\end{equation*}
where \(\varphi\) is a smooth function satisfying
\(\varphi^{\prime\prime}\geq K\).
\end{corollary}

An analogous spectral-gap rigidity theorem holds for
\(\mathrm{RCD}(K,\infty)\)-spaces \cite[Theorem~1.1]{GKKO2017}. Thus, in these
continuous settings, sharp eigenfunctions produce global product
directions.

However, our results show that the equality case in the Lichnerowicz estimate in the graph setting is fundamentally different. 

A related multiplicity-versus-rigidity phenomenon occurs for closed K\"ahler manifolds with a positive Ricci lower bound.
For the standard K\"ahler--Einstein metric \(\omega_{\mathbb{CP}^n}\), normalized by \(\mathrm{Ric}(\omega_{\mathbb{CP}^n})=\omega_{\mathbb{CP}^n}\), the model \(\mathbb{CP}^n\) has \(\lambda_1=1\) with multiplicity \(n^2+2n\).
A single $\lambda_1=1$ does not characterize \(\mathbb{CP}^n\), since \(\mathbb{CP}^k\times\mathbb{CP}^{\ell}\) may also attain \(\lambda_1=1\).
Chu, Wang, and Zhang \cite{CWZ2025} proved the optimal rigidity threshold in this setting:
\(\lambda_{n^2+3}=1\) forces the manifold to be \(\mathbb{CP}^n\), whereas \(\lambda_{n^2+2}=1\) is attained by \(\mathbb{CP}^{n-1}\times\mathbb{CP}^1\) and is therefore insufficient.

These parallels and differences motivate our main questions: what hypercube
structure survives
under equality in the graph Lichnerowicz estimate, and when does higher
spectral multiplicity promote this bundle structure to a Cartesian product or to full hypercube rigidity?

\subsection{Proof strategy for the structural theorem}

In the smooth weighted
Riemannian setting, let \(f\) be a nonconstant eigenfunction attaining equality
in the Lichnerowicz estimate. Equality in the Bochner argument forces
\(\operatorname{Hess}f=0\). Consequently, \(\nabla f\) is a parallel vector
field of constant $\card{\nabla f}>0$, and its flow determines a globally
consistent direction along which the space splits off a Gaussian
\(\mathbb{R}\)-factor. More generally, independent eigenfunctions attaining
equality in the Lichnerowicz estimate determine independent parallel directions
and hence a higher-dimensional Gaussian factor.

This mechanism has no direct graph analogue. A graph has no tangent bundle,
parallel vector fields, or geodesic flow that can be integrated. For a sharp
eigenfunction \(f\), the equality condition implies that \(\Gamma(f)\) is
constant, but
\[
  2\Gamma(f)(x)=\sum_{y\sim x}\bigl(f(y)-f(x)\bigr)^2
\]
controls only the total squared variation over all edges incident to \(x\). It
does not prevent an individual edge difference \(f(y)-f(x)\) from vanishing.
Thus a single sharp eigenfunction need not determine a globally consistent
discrete direction and, in general, does not produce even one Cartesian
\(H_1\)-factor. 
Therefore the correct global object must allow twisting and is a graph bundle
rather than necessarily a Cartesian product.

Our first step is therefore to use the entire eigenspace \(E_2\), rather than a
single eigenfunction. A basis of \(E_2\) defines the spectral map
\(\Phi:V(G)\to\mathbb{R}^{m_2(G)}\). We call an edge visible if its endpoints
have distinct images under \(\Phi\), and collapsed otherwise. Although \(\Phi\) depends on the chosen basis, the visible--collapsed
decomposition does not: an edge \(xy\) is collapsed precisely when
\(h(x)=h(y)\) for every \(h\in E_2\).  In our discrete setting, this canonical edge decomposition serves as an
analogue of the splitting and transverse directions in continuous splitting
arguments.

We then choose a generic \(f\in E_2\) which
separates the endpoints of every visible edge and is constant across every
collapsed edge.  The main technical difficulty is to turn the equality in the curvature
inequality into exact combinatorial information. A basic tool to compute the optimal curvature $\mathcal{K}_\infty(x)$ at $x$ is the curvature matrix $A_\infty(x)$ whose smallest eigenvalue is exact $\mathcal{K}_\infty(x)$. Sharpness gives
\(\Gamma_2(f)=2\Gamma(f)\), and hence both the midpoint identity
\eqref{eq:midpoint} and the curvature-matrix relation
\[
  A_\infty(x)(\nabla f)_x=2(\nabla f)_x.
\]

A key observation (Proposition~\ref{prop:residual-visible-exactness}) is that if $xy$ is visible, then for any another $z\sim x$, $(x;y,z)$ satisfies the exact square incidence. More precisely, if
\(xy\) is visible and \(z\neq y\) is another neighbour of \(x\), then \(y\)
and \(z\) are nonadjacent and admit a unique square completion \(w\), with
\(S_1(x)\cap S_1(w)=\{y,z\}\). This exact-square statement is the discrete
replacement for the parallelism supplied by the vanishing Hessian in the
continuous theory. 

Through this exact-square structure, we show that the visible subgraph is a
disjoint union of hypercubes and that adjacent components are connected by
hypercube isomorphisms. Consequently,
\(G\cong G^\prime\square_\sigma H_r\) is a hypercube bundle admitting a section,
where the base \(G^\prime\) satisfies \(\CD(2,\infty)\) and, when nontrivial,
\(\lambda_1(G^\prime)>2\).

\section{Preliminary}\label{Preliminary}
Throughout this article, all graphs are assumed to be simple, connected, unweighted, and locally finite. Let \(G=(V,E)\) be a graph with the non-normalized Laplacian
\[
  Lf(x)=\sum_{y\sim x}(f(y)-f(x)).
\]
The eigenvalues of $-L$ are ordered as
\[
  0=\lambda_0<\lambda_1\leq \lambda_2\leq\cdots,
\]
with multiplicity.  For $K>0$, put
\[
  E_K:=\ker(L+KI),\quad m_K:=\dim E_K.
\]
We write $d_G$ for the combinatorial distance, and
\[
        S_r(x):=\{y\in V:d_G(x,y)=r\},\quad
        B_r(x):=\{y\in V:d_G(x,y)\leq r\}.
\]
For any two adjacent vertices $x,y$, we write $xy\in E$ or $x\sim y$.

The $d$-dimensional hypercube $H_d$ is the graph with vertex set $ \mathbb{F}_2^d$, where two vertices are adjacent if and only if they differ in exactly one coordinate; $H_0$ denotes the one-vertex graph.

If $G_1$ and $G_2$ are graphs, their Cartesian product $G_1\square G_2$ has vertex set $V(G_1)\times V(G_2)$, with $(x_1,x_2)$ adjacent to $(y_1,y_2)$ precisely when either $x_1=y_1$ and $x_2\sim y_2$, or $x_2=y_2$ and $x_1\sim y_1$. Then $H_d, d\geq 1$ is isomorphic to the Cartesian product of $d$ copies of $K_2$.

The Bakry--\'Emery operators $\Gamma$ and $\Gamma_2$ for functions $f,g:V\to \R$ are defined by
\begin{align}
  2\Gamma(f,g)&=L(fg)-fLg-gLf,\label{eq:Gamma}
\\
  2\Gamma_2(f,g)&=L\Gamma(f,g)-\Gamma(f,Lg)-\Gamma(g,Lf).\label{eq:Gamma2}
\end{align}
Then we have
\begin{equation}\label{eq:Gamma_second}
     2\Gamma(f,g)(x)=\sum_{y\sim x}(f(y)-f(x))(g(y)-g(x)).
\end{equation}
We write $\Gamma(f):=\Gamma(f,f)$ and $\Gamma_2(f):=\Gamma_2(f,f)$ for short.

\begin{definition}\label{def:CD}
    Let $G=(V,E)$ be a graph. We say that $G$ satisfies the Bakry--\'Emery curvature dimension condition $\CD(K,\infty)$ at a vertex $x\in V$ if
    \[
  \Gamma_2(f)(x)\ge K\Gamma(f)(x)
\]
holds for any function $f:V\to\R$. The $\infty$-Bakry--\'Emery curvature $\mathcal{K}_{\infty}(x)$ at $x\in V$ is defined as
\[\mathcal{K}_{\infty}(x)=\sup\{K: G \text{ satisfies }\CD(K,\infty)\ \text{at}\ x\}.\]
We say $G$ satisfies $\CD(K,\infty)$ if it satisfies $\CD(K,\infty)$ at every $x\in V$.
\end{definition}


We recall two basic consequences of positive Bakry--\'Emery curvature: the
Lichnerowicz eigenvalue bound and the Myers diameter bound.

\begin{theorem}[Lichnerowicz bound, {\cite[Proposition~1.3]{LMP}}]
\label{thm:LichnerowiczEstimate}
     Assume $G$ satisfies $\CD(K,\infty)$ with $K>0$. Let $D=\max_{v\in V(G)}\deg(v)<\infty$. Then
     \begin{equation}
         \label{equ:LichnerowiczEstimate}
         \lambda_1\geq K.
     \end{equation}
\end{theorem}

\begin{theorem}[{\cite[Proposition~1.3 and Theorem~1.4]{LMP}}]
\label{thm:MyersBound}
    Assume $G$ satisfies $\CD(K,\infty)$ with $K>0$. Let $D=\max_{v\in V(G)}\deg(v)<\infty$. Let $\diam(G)$ be the diameter of $G$. Then $G$ satisfies Myers diameter bound
    \begin{equation}
    \label{eq:MyersBound}
        \diam(G)\leq\frac{2D}{K}.
    \end{equation}
    The equality holds if and only if $G$ is $H_D$.
\end{theorem}
From the above theorem, we know every graph with $\CD(K,\infty)$ for some $K>0$ is a finite graph.

\begin{lemma}[{\cite[Theorem~2.1 and Section~3.1]{LMP}}]\label{lem:LMP-sharp}
Assume $G$ satisfies $\CD(K,\infty)$ with $K>0$.  If $f\in E_K$, then $\Gamma (f)$ is constant on $V$ and $\Gamma_2 (f)=K\Gamma (f)$.  By polarization, $\Gamma(f,g)$ is constant for every $f,g\in E_K$.
\end{lemma}
\begin{proof}
For any $f\in E_K$, we compute by definition that
    \begin{equation*}
       2 \sum_{x\in V}(\Gamma_2(f)(x)-K\Gamma(f)(x))=\sum_{x\in V}L\Gamma(f)(x)=0.
    \end{equation*}
By assumption, we have $\Gamma_2(f)(x)-K\Gamma(f)(x)\geq 0$. Hence, $\Gamma_2(f)(x)-K\Gamma(f)(x)=0$ at any $x$. This forces $L\Gamma(f)(x)=0$. The only harmonic function over a connected finite graph is the constant function. So $\Gamma(f)$ is constant.
\end{proof}

A key tool for computing the Bakry--Émery curvature at a vertex is the so-called curvature matrix. For each vertex $x$ of a graph $G$, one can associate a symmetric $\deg(x)\times \deg(x)$ matrix $A_\infty(x)$, whose entries are determined by the local structure of the ball $B_2(x)$, such that
\begin{equation}
\mathcal{K}_\infty(x)=\lambda_{\min}\bigl(A_\infty(x)\bigr),
\end{equation}
where $\lambda_{\min}\bigl(A_\infty(x)\bigr)$ denotes the smallest eigenvalue of $A_\infty(x)$. This matrix can be obtained from the matrices representing $\Gamma_2$ and $\Gamma$ by taking a Schur complement, see \cite[Section 2]{CKLP}. Below we recall the definition of the curvature matrix $A_\infty(x)$, in a slightly reformulated form of \cite[Proposition 1.13(i)]{CKLP}.

\begin{definition}[Curvature matrices]\label{def:curMatrix}
    Let $G$ be a graph. For any vertex $x$, we label
\[
   S_1(x)=\{y_1,\ldots,y_{\mathrm{deg}(x)}\}.
\]
For $i\neq j$, set
\[
  \varepsilon_{ij}:=\begin{cases}1,&y_i\sim y_j,\\0,&y_i\not\sim y_j,
  \end{cases}
\]
and
\[
  \omega_{ij}:=\sum_{\substack{z\in S_2(x)\\ z\sim y_i,\ z\sim y_j}}
      \frac1{d_x^-(z)}\ \ \text{with}
  \ \ d_x^-(z):=|S_1(x)\cap S_1(z)|.
\]
Let
\begin{equation}
\label{equ:tandOmega}
 t_i:=\sum_{j\neq i}\varepsilon_{ij} \text{ and } \Omega_i:=\sum_{j\neq i}\omega_{ij}=\card{S_2(x)\cap S_1(y_i)}-\sum_{\substack{z\in S_2(x)\\ z\sim y_i}}\frac{1}{d_x^-(z)}.
\end{equation}
Then we have
\begin{align}
\label{CKLP-off}
      (A_\infty(x))_{ij}&=1-2\varepsilon_{ij}-2\omega_{ij}
  \ \ \text{for}\ \ i\neq j,
\\
\label{CKLP-diag}
       (A_\infty(x))_{ii}&=3-\frac{\mathrm{deg}(x)+\mathrm{deg}(y_i)}{2}+\frac52 t_i+2\Omega_i\\
\end{align}
\end{definition}
By \eqref{CKLP-off} and \eqref{CKLP-diag}, we derive that
\begin{equation}
    \label{CKLP-row}
      (A_\infty(x)\one)_i=2+\frac12 t_i+\frac{\mathrm{deg}(x)-\mathrm{deg}(y_i)}{2},
\end{equation}
where $\one$ is the $\deg(x)$-dimensional all-$1$ vector.

\begin{example}\label{ex:hypercube_cur_matrix}
    The hypercube graph $H_d$ has $\deg(x)=\deg(y_i)=d$, $\varepsilon_{ij}=0$ and $\omega_{ij}=1/2$ for every
$i\neq j$, hence \[A_\infty(x)=2I_d,\]
at any vertex $x$.
\end{example}

For any $f\in E_K$, it satisfies the following local midpoint formula.

\begin{lemma}[{\cite[Theorem~3.4]{LMP}}]\label{lem:LMP-midpoint}
Let $G$ satisfy $\CD(K,\infty)$ with $K>0$ and $f$ satisfy $\Gamma_2(f)(x)=K\Gamma (f)(x)$ at a vertex $x$.  For any $z\in S_2(x)$, we have
\begin{equation}\label{eq:midpoint}
    \frac{f(x)+f(z)}2
  =
  \frac{1}{d_x^{-}(z)}
  \sum_{y\in S_1(x)\cap S_1(z)}f(y).
\end{equation}
\end{lemma}

The next lemma is a useful observation.
For a function $f: V\to \mathbb{R}$, we define its gradient $\nabla f$ as a vector field on $V$ by \[(\nabla f)_x:=\left(f(y_1)-f(x),\dots,f(y_{\mathrm{deg}(x)})-f(x)\right)^T\in \mathbb{R}^{\mathrm{deg}(x)},\ \ \text{for any} \ \ x\in V(G).\]
\begin{lemma}
\label{lem:CKLP}
Let $G$ satisfy $\CD(K,\infty)$ with $K>0$ and $f$ satisfy $\Gamma_2(f)(x)=K\Gamma (f)(x)$ at a vertex $x$. Then, we have
\begin{equation}\label{eq:cur-eig}
    A_{\infty}(x)(\nabla f)_x=K(\nabla f)_x.
\end{equation}
\end{lemma}
\begin{proof}
    By Lemma \ref{lem:LMP-midpoint}, the function $f$ satisfies the midpoint formula \eqref{eq:midpoint}. Then we can verify directly by Definition \ref{def:curMatrix} that
    \[2(\Gamma_2 (f)(x)-K\Gamma (f)(x))=(\nabla f)_x^T(A_\infty(x)-KI)(\nabla f)_x.\]
    For a more structural proof of the above identity, we refer to \cite[Section 4]{HuLiu2026}. Noticing that $\Gamma_2(f)(x)-K\Gamma(f)(x)=0$ and $A_\infty(x)-KI$ is positive semidefinite, we  conclude \eqref{eq:cur-eig}.
\end{proof}

\section{Hypercube}
In this section, we collect several elementary facts about the hypercube graph \(H_d\) that will be used later. We first characterize the automorphism group of \(H_d\). 
\begin{lemma}
    Let $\tau\in \operatorname{Aut}(H_d)$. Then there exists $a\in\mathbb{F}_2^d$ and a permutation $\pi\in S_d$ such that
\begin{equation}
\label{equ:HypercubeAutomorphism}
    \tau(x)=a+\bigl(x_{\pi^{-1}(1)},\cdots,x_{\pi^{-1}(d)}\bigr).
\end{equation}
\end{lemma}
\begin{proof}
    Let \(a=\tau(0)\). Replacing \(\tau\) by \(T_a\circ\tau\), where \(T_a(x)=x+a\), we may assume that \(\tau(0)=0\). Then \(\tau\) permutes the neighbours of
\(0\), namely \(e_1,\ldots,e_d\). Hence there exists \(\pi\in S_d\) such that
\(\tau(e_i)=e_{\pi(i)}\). Now we prove the lemma by induction on the distance of $0$. Assume this lemma holds for all vertices of distance at most \(k\geq 1\) from \(0\). Let $S\subset \{1,\cdots, d\} $ with $|S|=k+1$. Let 
\begin{equation*}
    x_S=\sum_{i\in S}e_i\in S_{k+1}(0)\text{ and } x_{S,m}=\sum_{i\in S\setminus\{m\}}e_i\in S_{k}(0)
\end{equation*}
 for every $m\in S$. 

We have $\tau(x_S)\in S_{k+1}(0)$ and $\tau(x_{S,m})\in S_{k}(0)$ for every $m\in S$. By induction, for $m\in S$
\begin{equation*}
    \tau(x_{S,m})=\sum_{i\in S\setminus\{m\}}e_{\pi(i)}.
\end{equation*}
Since \(x_S\) is the unique vertex in \(S_{k+1}(0)\) adjacent to every
\(x_{S,m}\) and \(\tau\) preserves adjacency and distance
from \(0\), the vertex \(\tau(x_S)\) is the unique vertex in \(S_{k+1}(0)\)
adjacent to every \(\tau(x_{S,m})\), \(m\in S\).  Hence 
\begin{equation*}
    \tau(x_S)=\sum_{i\in S}e_{\pi(i)}.
\end{equation*}
This finishes the induction and proves the lemma.
\end{proof}

Next, we recall some spectral properties of the hypercube. 

When $d\geq 1$, $2$ is an eigenvalue of $H_d$. The spectrum of $H_d$ consists of $2k,0\leq k\leq d$, with multiplicity $\binom{d}{k}$. For $S\subseteq\{1,\cdots,d\}$, let $\chi_S(u):=(-1)^{\sum_{i\in S}u_i}$. A direct computation gives $L_{H_d}\chi_S=-2|S|\chi_S$. Let $W_k:=\operatorname{span}\{\chi_S: |S|=k\}$. Then $W_k$ is the $2k$-eigenspace of $H_d$. We obtain the orthogonal decomposition of $\mathbb{R}^{V(H_d)}$
\begin{equation}
\label{equ:DirectSum}
    \mathbb{R}^{V(H_d)}=\bigoplus_{k=0}^d{W_k}.
\end{equation}

\begin{lemma}
    \label{lem:Binary2Eigenvector}
    Let \(H_d\) be the \(d\)-dimensional hypercube with \(d\geq1\). If
    \(f\in\ker(L_{H_d}+2I)\) and \(f(V(H_d))\subseteq\{-1,1\}\), then
    \(f=\pm\chi_i\) for some \(1\leq i\leq d\).
\end{lemma}
\begin{proof}
Since \(f\in W_1(H_d)\), write $ f=\sum_{j=1}^d a_j\chi_j $.
The identity \(f^2=1\) gives
\[
    1=f^2=\sum_{j=1}^d a_j^2+2\sum_{j<k}a_ja_k\chi_{\{j,k\}}.
\]
By the linear independence of \(\chi_S\), we have
\(a_ja_k=0\) for \(j\ne k\), and the constant part gives
\(\sum_j a_j^2=1\). Therefore exactly one coefficient is nonzero and equal to
\(\pm1\), so \(f=\pm\chi_i\) for some \(i\).
\end{proof}

\begin{lemma}
\label{lem:large-holonomy-invariant-coordinate}
Let \(\Gamma\leq \operatorname{Aut}(H_r)\) be a subgroup. Let
\begin{equation*}
    W_1(H_r)^\Gamma=\{f\in W_1(H_r): f\circ \sigma =f\text{ for every $\sigma\in \Gamma$}\}.
\end{equation*}
If $\dim W_1(H_r)^\Gamma>\frac r2$, then there exists a function \(h\in W_1(H_r)^\Gamma\) such that
\[
    h(V(H_r))\subseteq\{-1,1\}.
\]
\end{lemma}
\begin{proof}
For \(\sigma\in\Gamma\), by \eqref{equ:HypercubeAutomorphism} there exist
\(a_\sigma\in\mathbb F_2^r\) and a permutation \(\pi_\sigma\in S_r\) such that
\[
    \chi_i\circ\sigma=(-1)^{(a_\sigma)_i}\chi_{\pi_\sigma^{-1}(i)}.
\]
Thus \(\Gamma\) acts on \(W_1(H_r)\) by signed permutations of
\(\chi_1,\ldots,\chi_r\). Let \(\mathcal O_1,\ldots,\mathcal O_k\) be the
orbits of the induced action on \(\{1,\ldots,r\}\), and let
\[
    U_{\mathcal O}:=\operatorname{span}\{\chi_i:i\in\mathcal O\}.
\]
Then \(W_1(H_r)=\bigoplus_{\mathcal O}U_{\mathcal O}\), and each
\(U_{\mathcal O}\) is \(\Gamma\)-invariant. Since the action on each orbit is
transitive up to signs, \(\dim U_{\mathcal O}^{\Gamma}\leq 1\).

Let \(c\) be the number of orbits \(\mathcal O\) such that
\(U_{\mathcal O}^{\Gamma}\neq0\). Then
\[
    c=\dim W_1(H_r)^\Gamma>\frac r2.
\]
If every such orbit had size at least \(2\), then
\[
    r\geq\sum_{U_{\mathcal O}^{\Gamma}\neq0}|\mathcal O|\geq 2c>r,
\]
a contradiction. Hence there is an orbit \(\mathcal O=\{i\}\) with
\(U_{\mathcal O}^{\Gamma}\neq0\). Therefore \(\chi_i\) is fixed by every
element of \(\Gamma\). Taking \(h=\chi_i\) gives
\(h\in W_1(H_r)^\Gamma\) and \(h(V(H_r))\subseteq\{-1,1\}\).
\end{proof}

\begin{lemma}
\label{lem:hypercube-unique-strict-local-maximum}
Assume $d\geq 1$. Let $f\in\ker(L_{H_d}+2I)$. Assume that $f(x)\neq f(y)$ for any edge $xy$. Then $f$ has a unique strict local maximum. Moreover, this local maximum is the unique global maximum of \(f\).
\end{lemma}

\begin{proof}
Let
\begin{equation*}
    f(x)=\sum_{i=1}^da_i(-1)^{x_i},\quad a_i\in\mathbb{R},\ x=(x_1,\cdots,x_d)\in \mathbb{F}_2^d.
\end{equation*}

If $x\in \mathbb{F}_2^d$, then any neighbor of $x$ can be written as $x+e_i$ where the $e_i$ is $i$-th unit vector of $\mathbb{F}_2^d$. Then
\begin{equation}
\label{equ:DifferenceOfEigenVectorOfHypercube}
    f(x+e_i)-f(x)=-2a_i(-1)^{x_i}.
\end{equation}
Since $f(x+e_i)\neq f(x)$, \eqref{equ:DifferenceOfEigenVectorOfHypercube} implies $a_i\neq 0$. Let $x_0$ be a maximum of $f$. Then \eqref{equ:DifferenceOfEigenVectorOfHypercube} gives $a_i(-1)^{{(x_0)}_i}<0$. So $(x_0)_i=0$ if $a_i>0$ and $(x_0)_i=1$ if $a_i<0$. Hence any strict local maximum of $f$ must be $x_0$.
Then $x_0$ is the only strict local maximum of $f$. On the other hand, any maximum of $f$ must also be a strict local maximum of $f$. Thus $x_0$ is the unique maximum of $f$.

\end{proof}

\section{Curvature-value rigidity}
In this section we prove that, for non-normalized Laplacian, equality in the Lichnerowicz estimate \eqref{equ:LichnerowiczEstimate} with \(K>0\) can occur only when \(K=2\).

\begin{theorem}[Curvature-value rigidity]\label{thm:first-sharp-K2}
Let \(G\) be a graph satisfying \(\CD(K,\infty)\) with \(K>0\).  If $\lambda_1=K$, then \(K=\mathcal{K}_\infty(x)=2\) for every vertex \(x\in V(G)\).
\end{theorem}
\begin{proof}
Recall $E_K:=\ker(L+KI)$ and choose $0\neq f\in E_K$.  For a vertex $x$, define the symmetric
bilinear form
\[
  \mathcal{B}_x(u,v):=\Gamma_2(u,v)(x)-K\Gamma(u,v)(x).
\]
The curvature condition implies that $\mathcal{B}_x$ is positive semidefinite. Moreover, Lemma~\ref{lem:LMP-sharp} gives $\mathcal{B}_x(f,f)=0$ for every vertex $x$. Hence, by the Cauchy--Schwarz inequality for semidefinite bilinear forms,
\[
    \mathcal{B}_x(f,h)=0
\]
for every test function $h$ and every vertex $x$.

Choose a vertex $x$ with $f(x)\neq 0$, and let $h=\one_{\{x\}}$ be the indicator function of $\{x\}$, i.e., $h(x)=1$ and $h(y)=0$ for all $y\neq x$. Then we have
\begin{equation}\label{eq:Lh}
 L h(x)=-\mathrm{deg}(x),\ \ \text{and}\ \ Lh(y)=1,\ \ \text{for all}\ \ y\sim x.
\end{equation}
By \eqref{eq:Gamma_second}, we derive
\begin{equation}\label{eq:Gammafh}
    2\Gamma(f,h)(x)=-Lf(x)\ \ \text{and}\ \ 2\Gamma(f,h)(y)=f(x)-f(y),\ \ \text{for all}\ \ y\sim x.
\end{equation}
Using \eqref{eq:Lh}, \eqref{eq:Gammafh}, and $Lf=-Kf$, we get
\[
\begin{aligned}
  2\mathcal{B}_x(f,h)
  =L\Gamma(f,h)(x)-\Gamma(f,Lh)(x)-K\Gamma(f,h)(x)
  =\frac{K(2-K)}2 f(x).
\end{aligned}
\]
Thus
\[
  0=4\mathcal{B}_x(f,h)=K(2-K) f(x).
\]
Since $K>0$ and $f(x)\neq0$, it follows that $K=2$.

Now we prove $\mathcal{K}_\infty(x)=2$. By Lemma~\ref{lem:LMP-sharp}, we have $\Gamma(f)$ is constant on $V$.  Since $f\neq 0$, so $\Gamma(f)(x)>0$.  We also have
\begin{equation*}
 0\leq \Gamma_2(f)-\mathcal{K}_\infty(x)\Gamma(f)=(K-\mathcal{K}_\infty(x))\Gamma(f)
\end{equation*}
Since $K\leq \mathcal{K}_\infty(x)$, we must have $\mathcal{K}_\infty(x)=K=2$.
\end{proof}

\section{Hypercube bundle structure of Lichnerowicz-sharp graphs}
Let $G$ be an unweighted and connected graph. We call $G$ Lichnerowicz-sharp if $G$ satisfies $\CD(2,\infty)$ and $\lambda_1(G)=2$. In this section, we prove for any regular Lichnerowicz-sharp graph, it has a bundle structure with fiber graph, an $r$-dimensional hypercube $H_r$. The following is the definition of graph bundle.

\begin{definition}[{\cite[Definiation 1]{LL2024}}]
\label{def:GraphBundle}
    Let $G$ and $F$ be two unweighted graphs. Let $E^{\mathrm{ori}}_{G}=\{(x,y): x\sim y \text{ in } G\}$ be the set of oriented edges of $G$. Let $\sigma: E^{\mathrm{ori}}_G\to \operatorname{Aut}(F), (x,y)\to \sigma_{xy}$ such that $\sigma_{yx}=\sigma_{xy}^{-1}$. Construct the graph $G\square_\sigma F$ by setting $V(G\square_\sigma F)=V(G)\times V(F)$ and $E(G\square_\sigma F)$
    \begin{enumerate}[label=\textup{(\roman*)}]
        \item $(x,u)\sim (y,\sigma_{xy}(u))$ if $x\sim y$ in $G$;
        \item $(x,u)\sim (x,v)$ if $u\sim v$ in $F$.
    \end{enumerate}
    We call $G\square_\sigma F$ the $F$-bundle over $G$ and refer to $G$ and $F$ as the base graph and the fiber graph, respectively.
\end{definition}
\begin{definition}
    Let $G\square_\sigma F$ be a $F$-bundle over $G$. A section of $G\square_\sigma F$ is a graph homomorphism $s: G\to G\square_\sigma F$ such that $\pi_1\circ s: V(G)\to V(G)$ is the identity map where $\pi_1$ is the projection to the first factor $V(G\square_\sigma F)\to V(G), (x,u)\to x$.
\end{definition}

From here to the end of this section, unless otherwise stated, we shall always assume that $G$ is a connected, $d$-regular, Lichnerowicz-sharp graph.

Let $E_2=\ker(L+2I)$ and $\phi_1,\ldots,\phi_{m}: V(G)\to \mathbb{R}$ be a basis of $E_2$, that is, $L\phi_i+2\phi_i=0$. Define a map $\Phi:V(G)\to \mathbb{R}^{m}$ by
\begin{equation*}
    \Phi(x)=(\phi_1(x),\cdots,\phi_m(x)).
\end{equation*}

We partition the edge set of $G$ into two classes according to whether
$\Phi(x)-\Phi(y)$ vanishes at the endpoints of an edge. We call an edge
$x\sim y$ of $G$ visible if $\Phi(x)\neq \Phi(y)$, and collapsed otherwise.
Let
\[
E_{\vis}=\{xy\in E(G): xy\text{ is visible}\},\quad
E_{\col}=\{xy\in E(G): xy\text{ is collapsed}\}.
\]
and $G_{\vis}=(V(G),E_{\vis})$ and $G_{\col}=(V(G),E_{\col})$.

Since every $f\in E_2$ is a linear combination of $\phi_1,\ldots,\phi_m$,
an edge $xy$ is collapsed if and only if $f(x)=f(y)$ for every
$2$-eigenvector $f$. Hence whether an edge is visible or collapsed is
independent of the choice of basis of $E_2$.

\begin{lemma}
\label{lem:OneVisibleEdge}
    $S_1^{G_{\vis}}(x)\neq\varnothing$ for any $x\in V(G)$.
\end{lemma}
\begin{proof}
    If not, assume $x_0\in V(G)$ with $S_1^{G_{\vis}}(x_0)=\varnothing$. We prove by induction that $\Phi$ is constant on $V(G)$. Since every edge incident to $x_0$ is collapsed, the result holds for $x$ with $d(x,x_0)\leq 1$. Now assume this holds for any $x$ with $d(x,x_0)\leq k, k\geq 1$. Let $x\in S_{k-1}^G(x_0), z\in S_{k+1}^G(x_0)$ with $d(x,z)=2$. By \eqref{eq:midpoint},
    \begin{equation*}
        \Phi(z)=\frac{2}{\card{S_1^G(x)\cap S_2^G(z)}}\sum_{y\in S_1^G(x)\cap S_2^G(z)}\Phi(y)-\Phi(x)=\Phi(x).
    \end{equation*}
    Thus this result holds for $x\in S_{k+1}^G(x_0)$ and $\Phi$ is constant on $V(G)$. Once $\Phi$ is constant, $L\Phi(x)=0$ contradicting $L\Phi(x)=-2\Phi(x)$.
\end{proof}

Let $G^\prime$ be the graph whose vertex set consists of the connected components of $G_{\vis}$, with two distinct components adjacent if and only if some collapsed edge of $G$ has one endpoint in each component. We will later show that both $G_{\vis}$ and $G_{\col}$ are regular, so the degree of $G_{\vis}$ is well defined. The following theorem is the main result of this section and gives a more precise version of Theorem~\ref{tm:Bundle-structure1} stated in the introduction.
\begin{theorem}
\label{thm:HypercubeBundle}
    Let $G$ be a connected, $d$-regular, Lichnerowicz-sharp graph. Then $G$ is a $H_r$-bundle whose base graph is $G^\prime$ and $r\geq 1$ is the degree of $G_{\vis}$. $G$ admits a section. If $G$ has at least one collapsed edge, then $G^\prime$ is a connected, regular, $\CD(2,\infty)$ graph with $\lambda_1(G^\prime)>2$. 
\end{theorem}
As for $r\geq m_2(G)$, it is a consequence of Proposition \ref{prop:EigenspaceOfBundle}.

\subsection{Exact square incidence}
\begin{lemma}
\label{lem:visible-generic-eigenfunction}
There exists $f\in E_2$ such that for every visible edge $uv$,
\begin{equation}\label{eq:visible-generic-property}
f(u)\ne f(v)
\end{equation}
and for every collapsed edge $uv$,
\begin{equation}
\label{eq:zero-difference-iff-collapsed}
    f(u)=f(v).
\end{equation}
\end{lemma}

\begin{proof}
For a visible edge $uv$, the linear functional
\begin{equation*}
\ell_{uv}:E_2\to\R,
\quad
\ell_{uv}(f)=f(u)-f(v)
\end{equation*}
is nonzero. Hence $\ker \ell_{uv}$ is a proper hyperplane in $E_2$. By Theorem~\ref{thm:MyersBound}, the graph is finite, so there are only finitely many visible edges. A finite union of proper hyperplanes cannot cover $E_2$. Choose $f$ outside this union. Then \eqref{eq:visible-generic-property} holds. If $uv$ is collapsed, then $\Phi(u)=\Phi(v)$, so every $f\in E_2$ satisfies $f(u)=f(v)$. This proves \eqref{eq:zero-difference-iff-collapsed}.
\end{proof}

By Lemma~\ref{lem:visible-generic-eigenfunction}, we  take a $f\in E_2$ such that $f(x)\neq f(y)$ for any visible edge $xy$ and $f(x)=f(y)$ for any collapsed edge.

\begin{proposition}
\label{prop:no-visible-triangle}
There are no vertices $x,y,z$ such that
\begin{equation*}
x\sim y,
\quad
x\sim z,
\quad
y\sim z
\end{equation*}
and $xy$ is a visible edge.
\end{proposition}

\begin{proof}
Assume that a triangle containing a visible edge exists. Let
\begin{equation*}
    \mathcal{T}=\{u\in V(G): \text{$u$ is an endpoint of a visible edge contained in a triangle}\}.
\end{equation*}
Choose $x\in\mathcal T$ such that $f(x)=\max_{u\in\mathcal T}f(u)$.

Write $S_1(x)=\{y_1,\ldots,y_d\}$ and let $t_i=\deg_{[S_1(x)]}(y_i)$. For each $i$ with $t_i>0$, if $xy_i$ is a collapsed edge,  then $f(x)=f(y_i)$. If $xy_i$ is visible, then $y_i\in\mathcal{T}$ and hence, $f(y_i)<f(x)$. Since $f\in E_2$, by Lemma \ref{lem:LMP-sharp} and Lemma~\ref{lem:CKLP}, we have $(A_{\infty}(x)-2I_d)(\nabla f)_x=0$. Combining with \eqref{CKLP-row}, we have
\begin{equation*}
\begin{aligned}
0
&=2\langle\mathbf{1},(A_{\infty}(x)-2I_d)(\nabla f)_x\rangle\\
&=2\langle(A_{\infty}(x)-2I_d)\mathbf{1},(\nabla f)_x\rangle\\
&=\sum_i t_{i}(f(y_i)-f(x)).
\end{aligned}
\end{equation*}

Since $x\in\mathcal{T}$, we take a visible edge $xy_m$ with $t_m>0$. Then $f(y_m)-f(x)<0$. This implies $\sum_i t_{i}(f(y_i)-f(x))<0$. It is a contradiction.
\end{proof}
\begin{definition}[Exact square incidence]
\label{def:pair-exact-square}
Let $x\in V(G)$ and let $y,z\in S_1(x)$ with $y\ne z$. We say that the triple $(x;y,z)$ satisfies exact square incidence if the following three conditions hold.
\begin{enumerate}[label=\textup{(\roman*)}]
\item The two neighbours are not adjacent;
\item There exists a unique $w\in S_2(x)$ such that $w\sim y, w\sim z$;
\item This vertex $w$ has exactly two back-neighbours at $x$:
\begin{equation*}
S_1(x)\cap S_1(w)=\{y,z\}.
\end{equation*}
\end{enumerate}
We call $w$ the exact square completion of the triple $(x;y,z)$.
\end{definition}
We say a vertex $x$ is good if for any visible edge $xy$ and any $z\in S_1(x)\setminus\{y\}$, $(x;y,z)$ satisfies exact square incidence. A key observation is that every vertex is good.

\begin{lemma}
\label{lem:visible-star-transfer}
Let $xy$ be a visible edge. Assume that $y$ is good. Then for every $z\in S_1(x)\setminus\{y\}$, the triple $(x;y,z)$ satisfies exact square incidence.
\end{lemma}

\begin{proof}
Since $xy$ is visible, Proposition~\ref{prop:no-visible-triangle} gives $y\not\sim z$ for every $z\in S_1(x)\setminus\{y\}$. This proves \textup{(i)}.

Since $y$ is good, for every $u\in S_1(y)\setminus\{x\}$, the triple $(y;x,u)$ satisfies exact square incidence. So there exists a unique exact square completion $\sigma_y(u)\in S_2(y)$ for $(y;x,u)$. Thus we obtain a map
\begin{equation*}
\sigma_y:S_1(y)\setminus\{x\}\to S_1(x)\setminus\{y\}.
\end{equation*}
This map is injective. Indeed, if $\sigma_y(u)=\sigma_y(u')$ with $u\ne u'$, then $\{x,u,u^\prime\}\subseteq S_1(y)\cap S_1(\sigma_y(u))$, contradicting $d_y^{-}(\sigma_y(u))=2$. Since $|S_1(x)\setminus\{y\}|=\card{S_1(y)\setminus\{x\}}=d-1$, $\sigma_y$ is bijective.

Fix $z\in S_1(x)\setminus\{y\}$. By bijectivity, there is a unique $w\in S_1(y)\setminus\{x\}$ such that $\sigma_y(w)=z$. Because $y$ is good, the vertices $x$ and $w$ are not adjacent. Therefore $w\in S_2(x)$.
If $w'\in S_2(x)$ also satisfies $w'\sim y$ and $w'\sim z$, then $\sigma_y(w^\prime)=z$ and injectivity gives $w'=w$. This proves \textup{(ii)}.

Finally, suppose $w$ were adjacent to some $z'\in S_1(x)\setminus\{y,z\}$. Then $\sigma_y^{-1}(z)=\sigma_y^{-1}(z^\prime)=w$ contradicting $\sigma_y(w)=z$ and $z^\prime\neq z$.  Hence $S_1(x)\cap S_1(w)=\{y,z\}$. This proves \textup{(iii)}.

\end{proof}

For any $x\in V(G)$, we define the following three sets and quantities:
\begin{equation*}
\begin{aligned}
S_1^-(x)&=\{y\sim x:f(y)<f(x)\},\ &r=\card{S_1^-(x)},\\
S_1^0(x)&=\{y\sim x:f(y)=f(x)\},\ &s=\card{S_1^0(x)},\\
S_1^+(x)&=\{y\sim x:f(y)>f(x)\},\ &m=\card{S_1^+(x)}.
\end{aligned}
\end{equation*}
By the construction of $f$, every vertex in $S_1^-(x)\cup S_1^+(x)$ is joined to $x$ by a visible edge, and every vertex in $S_1^0(x)$ is joined to $x$ by a collapsed edge.

Assume every vertex in $S_1^-(x)$ is good. For each $n\in S_1^-(x)$ and each $u\in S_1(x)\setminus\{n\}$, Lemma~\ref{lem:visible-star-transfer} gives a unique exact square completion for $(x;n,u)$. Let $K_x$ denote the set of $v\in S_2(x)$ which is the unique exact square completion for some $n\in S_1^-(x)$ and $u\in S_1(x)\setminus\{n\}$. If $a\in S_2(x)\setminus K_x$, then $a$ has no back-neighbour in $S_1^{-}(x)$. Indeed, if $a\sim n$ for some $n\in S_1^-(x)$, then $a$ is the exact square completion for $(x;n,\sigma_n(a))$, where $\sigma_n$ is the bijection constructed in the proof of Lemma~\ref{lem:visible-star-transfer}. Hence $S_1(x)\cap S_1(a)\subseteq S_1^0(x)\cup S_1^+(x)$. Define
\begin{equation*}
R_x^+=\{a\in S_2(x)\setminus K_x:S_1(a)\cap S_1^+(x)\ne\varnothing\}.
\end{equation*}
For $a\in R_x^+$, set
\begin{equation*}
    \begin{aligned}
        H_a&=S_1(a)\cap S_1^+(x),
\quad h_a=\card{H_a},\\
Z_a&=S_1(a)\cap S_1^0(x),
\quad z_a=\card{Z_a}.\\
    \end{aligned}
\end{equation*}
We also define $\rho_a=h_a+z_a$ and
\begin{equation*}
B_a=\sum_{p\in H_a}{(f(p)-f(x))}.
\end{equation*}
Using \eqref{eq:midpoint} for $x,a$ and since $S_1(a)\cap S_1^{-}(x)=\varnothing$, we obtain
\begin{equation}\label{eq:f-a-via-B-a}
f(a)-f(x)=\frac{2B_a}{\rho_a}.
\end{equation}
In particular, $f(a)>f(x)$.

Let $\mathcal P_x=
\bigl\{\{p,u\}:p\in S_1^+(x),\ u\in S_1^+(x)\cup S_1^0(x),\ u\ne p\bigr\}$.
Its cardinality is
\begin{equation*}
N_x=\card{\mathcal P_x}=\binom{m}{2}+ms.
\end{equation*}

\begin{lemma}
\label{lem:residual-existence}
Assume every vertex in $S^{-}_1(x)$ is good. Then every pair $e\in\mathcal P_x$ has at least one common neighbour in $R_x^+$.
\end{lemma}

\begin{proof}
Let $e=\{p,u\}\in\mathcal P_x$ with $p\in S_1^+(x)$. Since $xp$ is visible, by Proposition~\ref{prop:no-visible-triangle}, $xp$ couldn't be contained in a triangle. Thus $d(p,u)=2$. Suppose that $p$ and $u$ have no common neighbour in $R_x^+$. A common neighbour different from $x$ lies in $S_2(x)$. It cannot lie in $K_x$. Otherwise, its back-degree is exactly two against having at least three back-neighbours. Hence the only common neighbour of $p$ and $u$ is $x$. Applying \eqref{eq:midpoint} to $p,u$ gives
\begin{equation}\label{eq:existence-contradiction-midpoint}
\frac{f(p)+f(u)}2=f(x).
\end{equation}
But $f(p)>f(x)$ and $f(u)\ge f(x)$, contradicting \eqref{eq:existence-contradiction-midpoint}. Therefore such a common neighbour exists.
\end{proof}

For $a\in R_x^+$, the number of pairs in $\mathcal{P}_x$, such that $a$ is the exact square completion for the pair and $x$, is $I_a=\binom{h_a}{2}+h_az_a$. Let
\begin{equation}
I=\sum_{a\in R_x^+}I_a.
\end{equation}
Lemma~\ref{lem:residual-existence} implies $I\geq N_x$.

\begin{lemma}
    \label{lem:zero-fan-exclusion}
Assume every vertex in $S^{-}_1(x)$ is good.  For every $a\in R_x^+$, $z_a\le1$.
\end{lemma}

\begin{proof}
Assume instead that there are two distinct vertices $c,c'\in Z_a$. By \eqref{eq:f-a-via-B-a},
\begin{equation*}
f(a)>f(x)=f(c)=f(c').
\end{equation*}
If $c\sim c'$, then $cac^\prime$ is a triangle. But the edge $ca$ is visible contradicting Proposition~\ref{prop:no-visible-triangle}.

Thus $c\not\sim c'$, so $d(c,c')=2$. We apply \eqref{eq:midpoint} to $c,c'$. We claim that every common neighbour $q$ of $c,c'$ other than $a$ satisfies
\begin{equation}\label{eq:common-neighbor-nonnegative}
f(q)\ge f(x).
\end{equation}
Indeed, if $q=x$, this is clear. If $q\in S_1(x)$ and $q\ne x$, then $q$ cannot be visible from $x$, because the triangle $xqc$ would contain the visible edge $xq$. Hence $q\in S_1^0(x)$ and $f(q)=f(x)$. If $q\in S_2(x)$, then $q$ cannot lie in $K_x$, since every vertex in $K_x$ has exactly two back-neighbours at $x$ and at least one of them has lower $f$-value, while $q$ is adjacent to the two zero back-neighbours $c,c'$. Hence $q\not\in K_x$ and all back-neighbours of $q$ at $x$ lie in $S_1^0(x)\cup S_1^+(x)$.  Applying \eqref{eq:midpoint} to the pair $x,q$ gives \eqref{eq:common-neighbor-nonnegative}.

Now $a$ is one common neighbour of $c,c'$ and satisfies $f(a)>f(x)$, while all other common neighbours satisfy \eqref{eq:common-neighbor-nonnegative}.
Therefore by \eqref{eq:midpoint}
\[
\begin{aligned}
\frac{f(c)+f(c')}{2}
&=
\frac{1}{\lvert S_1(c)\cap S_1(c')\rvert}
\left(
\sum_{\substack{w\sim c,\ w\sim c^\prime\\ w\neq a}} f(w)+f(a)
\right)>f(x).
\end{aligned}
\]
This contradicts $f(x)=f(c)=f(c^\prime)$.
\end{proof}

\begin{lemma}
\label{lem:visible-diagonal-estimate}
Assume every vertex in $S^{-}_1(x)$ is good.  One has
\begin{align}
         \label{eq:degree-count-DH}
    \sum_{z\in R_x^+}h_z&=m(m+s-1);\\
    \label{eq:diagonal-estimate}
    \sum_{z\in R_x^+}\frac{h_z}{\rho_z}&\le \frac{m(m+s-1)}{2}. 
\end{align}
\end{lemma}

\begin{proof}
Fix $p\in S_1^+(x)$. Since $xp$ is visible, Proposition~\ref{prop:no-visible-triangle} implies that $p$ is not adjacent to any vertex of $S_1(x)\setminus\{p\}$. The neighbours of $p$ consist of $x$, the $r$ vertices in $K_x$ adjacent to pairs $p$ and $n$ with $n\in S_1^-(x)$, and the remaining neighbours in $R_x^+$. Hence the number of neighbours of $p$ in $R_x^+$ is
\begin{equation}\label{eq:p-residual-degree}
d-1-r=m+s-1.
\end{equation}
Summing \eqref{eq:p-residual-degree} over all $p\in S_1^+(x)$ gives \eqref{eq:degree-count-DH}.

Since $A_\infty(x)\succeq2I_d$, we have $(A_\infty(x))_{pp}\ge2$.  By \eqref{CKLP-diag} and $t_p=0$, we have
\begin{equation*}
    (A_{\infty}(x))_{pp}=3-d+2\Omega_p=3-d+2\card{S_2(x)\cap S_1(p)}-2\sum_{\substack{z\in S_2(x)\\ z\sim p}}\frac{1}{d^-_x(z)}.
\end{equation*}
Together with $\card{S_2(x)\cap S_1(p)}=d-1$, it gives
\begin{equation*}
\begin{aligned}
    \sum_{\substack{z\in S_2(x)\\ z\sim p}}\frac1{d_x^-(z)}=&\sum_{\substack{z\in K_x\\ z\sim p}}\frac1{d_x^-(z)}+\sum_{\substack{z\in R_x^+\\ z\sim p}}\frac1{d_x^-(z)}=\frac{r}{2}+\sum_{\substack{z\in R_x^+\\ z\sim p }}\frac1{d_x^-(z)}\le \frac{d-1}{2}.
\end{aligned}
\end{equation*}
Therefore
\begin{equation}\label{eq:p-diagonal-residual}
\sum_{\substack{z\in R_x^+\\ z\sim p}}\frac1{\rho_z}
\le
\frac{d-1-r}{2}
=
\frac{m+s-1}{2}.
\end{equation}
Here $\rho_z=d^{-}_x(z)$ because every back-neighbour of $z$ lies in $H_z\cup Z_z$. Summing \eqref{eq:p-diagonal-residual} over $p\in S_1^+(x)$ gives \eqref{eq:diagonal-estimate}.
\end{proof}

\begin{lemma}
\label{lem:amhm-uniqueness}
Assume every vertex in $S^{-}_1(x)$ is good.  Every pair $e\in\mathcal P_x$ has exactly one common neighbour in $R_x^+$.
\end{lemma}

\begin{proof}
For $e=\{a,b\}\in\mathcal P_x$, let
\begin{equation*}
W_e=\{w\in R_x^+:w\sim a,\ w\sim b\},
\quad
\eta_e=\card{W_e}.
\end{equation*}
By Lemma~\ref{lem:residual-existence}, $\eta_e\ge1$. Define
\begin{equation}
\gamma_e=(f(a)-f(x))+(f(b)-f(x))>0.
\end{equation}
The common neighbours of $a$ and $b$ are exactly $\{x\}\cup W_e$. Applying \eqref{eq:midpoint} to the pair $a,b$ gives
\begin{equation}
\sum_{w\in W_e}\bigl(f(w)-f(x)\bigr)=\frac{\eta_e+1}{2}\gamma_e.
\end{equation}
Dividing by $\gamma_e$, summing over $e\in\mathcal P_x$ and using \eqref{eq:f-a-via-B-a}, we obtain
\begin{equation}\label{eq:global-amhm-identity}
\sum_{w\in R_x^+}\frac{2B_w}{\rho_w}
\sum_{\substack{\{a,b\}\subset H_w\cup Z_w\\ \{a,b\}\cap H_w\ne\varnothing}}
\frac1{\gamma_{\{a,b\}}}
=
\frac12\bigl(I+N_x\bigr).
\end{equation}
Here $I=\sum_{\{a,b\}\in\mathcal P_x}\eta_{\{a,b\}}=\sum_{w\in R_x^+} I_w$.

Fix $w\in R_x^+$. By Lemma~\ref{lem:zero-fan-exclusion}, $z_w\le1$. When $I_w>0$, we have
\begin{equation*}
\begin{aligned}
     \frac{1}{I_w}\sum_{\substack{\{a,b\}\subset H_w\cup Z_w\\ \{a,b\}\cap H_w\neq\varnothing}}\bigl(f(a)-f(x)+f(b)-f(x)\bigr)=\frac{2(h_w+z_w-1)}{h_w(h_w+2z_w-1)}B_w.
\end{aligned}
\end{equation*}
Whenever $z_w=0$ or $z_w=1$, we have $\frac{(h_w+z_w-1)}{h_w(h_w+2z_w-1)}=\frac{1}{\rho_w}$. Thus
\begin{equation*}
    \frac{1}{I_w}\sum_{\substack{\{a,b\}\subset H_w\cup Z_w\\ \{a,b\}\cap H_w\ne\varnothing}}
\gamma_{\{a,b\}}=\frac{2B_w}{\rho_w}.
\end{equation*}
Therefore the arithmetic mean--harmonic mean inequality gives
\begin{equation}\label{eq:local-amhm-bound}
\frac{2B_w}{\rho_w}
\sum_{\substack{\{a,b\}\subset H_w\cup Z_w\\ \{a,b\}\cap H_w\ne\varnothing}}
\frac1{\gamma_{\{a,b\}}}
\ge I_w.
\end{equation}
Summing \eqref{eq:local-amhm-bound} over $w$ and using \eqref{eq:global-amhm-identity} yields $I\leq N_x$. Together with $I\geq N_x$, we have $I=N_x$.
Since every $\eta_{\{a,b\}}\ge1$ and $\sum_{\{a,b\}\in\mathcal P_x}\eta_{\{a,b\}}=N_x=\card{\mathcal P_x}$, every $\eta_{\{a,b\}}$ equals one. Thus every pair in $\mathcal{P}_x$ has exactly one common neighbour in $R_x^+$.
\end{proof}

\begin{lemma}
\label{lem:backdegree-two}
Assume every vertex in $S^{-}_1(x)$ is good.  We have $\rho_w=2$ for $w\in R_x^+$.
\end{lemma}

\begin{proof}
Since $I=N_x$, \eqref{eq:local-amhm-bound} must take an equality after summing over $w$. Equality holds in the arithmetic mean--harmonic mean inequality for every $w$ with $I_w>0$. Therefore, for such a $w$, all numbers $\gamma_{\{a,b\}}$ are equal for any $\{a,b\}\subset H_w\cup Z_w$ with $\{a,b\}\cap H_w\neq\varnothing$.

If $z_w=1$ and $h_w\ge2$, there exist $a,b\in H_w, c\in Z_w$. Then we have $\gamma_{\{a,b\}}=\gamma_{\{a,c\}}$. This gives $f(c)=f(b)$. It is a contradiction. Since $w\in R_x^+$, this gives $h_w=1$ and hence $\rho_w=2$.

 Let
\begin{equation*}
    \begin{aligned}
        \mathcal F&=\{w\in R_x^+:z_w=0,\ h_w\ge3\},\\
        \mathcal D&=\{w\in R_x^+:z_w=0,\ h_w=1\}.
    \end{aligned}
\end{equation*}
It remains to show $\mathcal{F}=\mathcal{D}=\varnothing$.
Let
\begin{equation}
D_H=\sum_{\substack{e\in\mathcal{P}_x}}\card{e\cap S_1^+(x)}=2\binom m2+ms=m(m+s-1).
\end{equation}
Because every desired pair $e$ has a unique $w\in R_x^+$ adjacent to them, counting $\card{e\cap S_1^+(x) }$ through $w$ gives
\begin{equation}\label{eq:high-appearance-through-w}
D_H=\sum_{w\in R_x^+}h_w(\rho_w-1).
\end{equation}
On the other hand, \eqref{eq:degree-count-DH} gives
\begin{equation}\label{eq:high-incidence-degree}
D_H=\sum_{w\in R_x^+}h_w.
\end{equation}
Subtracting \eqref{eq:high-incidence-degree} from \eqref{eq:high-appearance-through-w}, gives
\begin{equation}\label{eq:dangling-fan-balance}
\card{\mathcal D}=
\sum_{w\in\mathcal F} h_w(h_w-2).
\end{equation}

Now use \eqref{eq:diagonal-estimate}, which says
\begin{equation}\label{eq:diagonal-minus-DH}
\begin{aligned}
    &\sum_{w\in R_x^+}\left(\frac{h_w}{\rho_w}-\frac{h_w}{2}\right)\\
    =&\sum_{w\in\mathcal{F}}\left(1-\frac{h_w}{2}\right)+\sum_{w\in\mathcal{D}}\frac{1}{2}\\
    =&\frac{\card{\mathcal D}}{2}
-
\frac12\sum_{w\in\mathcal F}(h_w-2)\leq0.
\end{aligned}
\end{equation}
Using \eqref{eq:dangling-fan-balance}, we have
\begin{equation}\label{eq:fan-positive-contradiction}
\sum_{w\in R_x^+}\left(\frac{h_w}{\rho_w}-\frac{h_w}{2}\right)
=
\frac12\sum_{w\in\mathcal F}(h_w-1)(h_w-2).
\end{equation}
If $\mathcal F$ were nonempty, the right side of \eqref{eq:fan-positive-contradiction} would be strictly positive, contradicting the non-positivity in \eqref{eq:diagonal-minus-DH}. Hence $\mathcal F=\varnothing$. \eqref{eq:dangling-fan-balance} then gives $\mathcal D=\varnothing$.
Therefore every $w\in R_x^+$ has $\rho_w=2$.
\end{proof}

\begin{proposition}
\label{prop:residual-visible-exactness}
Every vertex $x$ is good.
\end{proposition}
\begin{proof}
Assume $f(x_1)\leq \cdots\leq f(x_n)$. We prove the proposition by induction on $x_i$. Assume for $1\leq i\leq k$, $x_i$ is good.  Let $x_{k+1}y$ be a visible edge and $z\in S_1(x_{k+1})\setminus\{y\}$. If one of $y,z$ is in $S_1^{-}(x_{k+1})$, then by Lemma \ref{lem:visible-star-transfer}, $(x_{k+1};y,z)$ satisfies exact square incidence. Now assume $y\in S_1^+(x_{k+1}), z\in S_1^0(x_{k+1})\cup S_1^+(x_{k+1})$. Proposition~\ref{prop:no-visible-triangle} gives $y\not\sim z$. Lemma~\ref{lem:amhm-uniqueness} gives a unique $w\in R_{x_{k+1}}^+$ adjacent to both $y$ and $z$. Lemma~\ref{lem:backdegree-two} gives $S_1(x_{k+1})\cap S_1(w)=\{y,z\}$. Therefore $(x_{k+1};y,z)$ satisfies exact square incidence. Thus $x_{k+1}$ is also good.
\end{proof}

As a consequence of Proposition \ref{prop:residual-visible-exactness}, we prove a new rigidity theorem for hypercubes.

\begin{theorem}[Hypercube rigidity]
\label{thm:HypercubeRigidity}
    Let $G$ be a connected, $d$-regular Lichnerowicz-sharp graph. If every edge of $G$ is visible, then $G\cong H_d$.
\end{theorem}
\begin{proof}
By Lemma \ref{lem:visible-generic-eigenfunction}, we take a $f\in E_2$ such that $f(u)\neq f(v)$ for any edge $uv$. For $x\in V(G)$ define
\begin{equation*}
A_x(f)=\{f(u)-f(x):u\sim x\}
\end{equation*}
as a multiset, and define
\begin{equation*}
\nu(x)=\card{\{u\sim x:f(u)<f(x)\}}.
\end{equation*}

Let $x\sim y$. For any $u\in S_1(x)\setminus\{y\}$, by Proposition \ref{prop:residual-visible-exactness}, $(x;y,u)$ satisfies exact square incidence. Thus there exists a unique square completion $v$ for $(x;y,u)$. Since $S_1(x)\cap S_1(v)=\{y,u\}$, applying \eqref{eq:midpoint} to $x,v$ gives
\begin{equation*}
f(v)-f(y)=f(u)-f(x).
\end{equation*}
Since $\card{S_1(x)}=\card{S_1(y)}=d$, these vertices exhaust
$S_1(y)\setminus\{x\}$. Therefore
\begin{equation}\label{eq:edgewise-rigidity-sign-flip}
A_y(f)=\bigl(A_x(f)\setminus\{f(y)-f(x)\}\bigr)\cup\{f(x)-f(y)\}
\end{equation}
as multisets. Because $f(x)-f(y)\ne0$, \eqref{eq:edgewise-rigidity-sign-flip} implies
\begin{equation}\label{eq:edgewise-rigidity-nu-change}
\lvert\nu(y)-\nu(x)\rvert=1
\end{equation}
for every edge $xy$.

Let $m\in V(G)$ be a point where $f$ attains its minimum, and let
$M\in V(G)$ be a point where $f$ attains its maximum. Since $f(u)\neq f(m), f(v)\neq f(M)$ for any $u\sim m, v\sim M$, we have $f(u)>f(m), f(v)<f(M)$. Thus $m\neq M$ and $\nu(m)=0, \nu(M)=d$.
\eqref{eq:edgewise-rigidity-nu-change} implies that every path from
$m$ to $M$ has length at least $d$. Hence $\diam(G)\ge d$. On the other hand, \eqref{eq:MyersBound} gives $\diam(G)\le d$. Therefore $\diam(G)=d$.
By Theorem \ref{thm:MyersBound}, the diameter rigidity of hypercube gives $G\cong H_d$.
\end{proof}

\subsection{Bundle structure}

Let $xy$ be an edge of $G$. By Lemma~\ref{lem:OneVisibleEdge}, both $S_1^{G_{\vis}}(x)$ and $S_1^{G_{\vis}}(y)$ are nonempty. We define a map $T_{xy}:S_1^{G_{\vis}}(x)\to S_1^{G_{\vis}}(y)$ as follows. If $xy$ is visible, set $T_{xy}(y)=x$. Take $z\in S_1^{G_{\vis}}(x)\setminus\{y\}$. Proposition~\ref{prop:residual-visible-exactness} gives a unique exact square completion $w$ for $(x;z,y)$. By
\eqref{eq:midpoint},
\begin{equation*}
    \Phi(x)-\Phi(z)=\Phi(y)-\Phi(w)\neq 0.
\end{equation*}
Hence $yw\in E_{\vis}$. Then $w\in S_1^{G_{\vis}}(y)$ and define $T_{xy}(z)=w$.
\begin{lemma}
\label{lem:TransportMap}
    $T_{xy}$ is a bijection.
\end{lemma}
\begin{proof}
  We show $T_{xy}$ is injective. Take $z\in S_1^{G_{\vis}}(x)\setminus\{y\}$. Let $w=T_{xy}(z)$. Assume $T_{xy}(z)=T_{xy}(z')=w$. Because $w\in S_2^G(x)$ we have $w\neq x$. Thus if $xy$ is visible, $T_{xy}^{-1}(x)=\{y\}$ and $z^\prime\neq y$. If $xy$ is collapsed, $y\not\in S_1^{G_{\vis}}(x)$. We also have $z^\prime\neq y$. Thus $S_1^G(x)\cap S_1^G(w)=\{y,z\}$ forces $z=z^\prime$. Hence $\card{S_1^{G_{\vis}}(x)}\leq |S_1^{G_{\vis}}(y)|$. Repeating the same argument with the roles of $x$ and $y$ interchanged gives $\card{S_1^{G_{\vis}}(y)}\leq |S_1^{G_{\vis}}(x)|$. So $\card{S_1^{G_{\vis}}(x)}= |S_1^{G_{\vis}}(y)|$. This implies $T_{xy}$ is also surjective. Then $T_{xy}$ is a bijection.
\end{proof}

Next, we characterize the structure of $G_{\vis}$. It follows that in fact, every connected component of $G_{\vis}$ is a hypercube.
\begin{theorem}
\label{thm:G_vis}
   $G_{\vis}$ is a regular graph and every connected component of $G_{\vis}$ is isomorphic to an $r$-dimensional hypercube where $r$ is the degree of $G_{\vis}$ and $r\geq 1$.
\end{theorem}
\begin{proof}
By Lemma~\ref{lem:TransportMap}, we know $\card{S_1^{G_{\vis}}(x)}=\card{S_1^{G_{\vis}}(y)}$ for $x\sim y$. Since $G$ is connected, we know $\card{S_1^{G_{\vis}}(x)}$ is constant on $V(G)$. Thus $G_{\vis}$ is a regular graph. Let $r=\card{S_1^{G_{\vis}}(x)}$. By Lemma~\ref{lem:OneVisibleEdge}, $S_1^{G_{\vis}}(x)\neq \varnothing$. So $r\geq 1$.

Let $C$ be a connected component of $G_{\vis}$. When $r=1$, $C$ is a $1$-regular graph, that is $C\cong H_1$. Now assume $r\geq 2$. Let $y,z\in V(C)$ be two distinct vertices with $xy,xz\in E_{\vis}$. By Proposition~\ref{prop:residual-visible-exactness}, $y\not\sim z$ and there exists a unique $w\in S_2^G(x)$ with $S_1^G(x)\cap S_1^G(w)=\{y,z\}$. By \eqref{eq:midpoint}, we have
\begin{equation*}
\Phi(x)-\Phi(y)=\Phi(z)-\Phi(w)\neq0,\quad \Phi(x)-\Phi(z)=\Phi(y)-\Phi(w)\neq0.
\end{equation*}
This implies $yw,zw\in E_{\vis}$. Thus we also have $S_1^{G_{\vis}}(x)\cap S_1^{G_{\vis}}(w)=\{y,z\}$. So $\varepsilon^{G_{\vis}}_{yz}=0$ and $\omega_{yz}^{G_{\vis}}=\frac{1}{2}$. \eqref{CKLP-diag} and \eqref{CKLP-off} give
\begin{equation*}
    \bigl(A_{\infty}^{G_{\vis}}\bigr)_{yy}=2,\quad \bigl(A_{\infty}^{G_{\vis}}\bigr)_{yz}=0.
\end{equation*}
Hence $A_{\infty}^{G_{\vis}}(x)=2I_{r}$. This shows $C$ itself is a $\CD(2,\infty)$ graph. A direct computation shows $f|_{V(C)}\in E_2(C)$ for any $f\in E_2(G)$. By \eqref{equ:LichnerowiczEstimate}, we have $\lambda_1(C)=2$. Since $C$ is also $r$-regular and its every edge is visible, by Theorem~\ref{thm:HypercubeRigidity}, we conclude $C\cong H_r$.
\end{proof}

\begin{lemma}
\label{lem:other-visible-component-neighbour-count-constant}
Let $C$ and $D$ be two connected components of $G_{\vis}$, allowed to be equal. For
$u\in V(C)$,
\begin{equation*}
m_{CD}(u)=\card{S_1^{G_{\col}}(u)\cap V(D)}.
\end{equation*}
Then $m_{CD}(u)$ is constant on $V(C)$.
\end{lemma}

\begin{proof}
It is enough to prove that $m_{CD}(u)=m_{CD}(u')$ whenever $u,u'\in V(C)$ and $uu'$ is a visible edge. Fix such an edge $uu'$. We first assume $m_{CD}(u)>0$. Define a map $T : S_1^{G_{\col}}(u)\cap V(D)\to S_1^{G_{\col}}(u^{\prime})\cap V(D)$ as follows.

Take $y\in S_1^{G_{\col}}(u)\cap V(D)$.  Since $uy$ is collapsed, $\Phi(y)=\Phi(u)$. Apply Proposition~\ref{prop:residual-visible-exactness} to $(u;u^\prime,y)$. We obtain a unique exact square completion $y^{\prime}$ for $(u;u^{\prime},y)$. By \eqref{eq:midpoint}, we have
\begin{equation*}
    \Phi(y^{\prime})=\Phi(u^{\prime}), \quad \Phi(y')-\Phi(y)=\Phi(u')-\Phi(u)\ne0.
\end{equation*}
Therefore $yy'$ is a visible edge. Since $y\in V(D)$, it follows that $y^\prime\in S_1^{G_{\col}}(u^\prime)\cap V(D)$. We define $T(y)=y^{\prime}$.

This map is injective. Indeed, assume two vertices $y_1,y_2\in S_1^{G_{\col}}(u)\cap V(D)$ give the same image $y'$. Since $uy_2\in E_{\col}$ while $uu'\in E_{\vis}$, we have $y_2\ne u'$. Hence
\[
y_2\in S_1^G(u)\cap S_1^G(y')=\{u',y_1\}
\]
forces $y_2=y_1$. Therefore $T$ is injective, and so
$m_{CD}(u)\le m_{CD}(u')$.

By applying the same argument with $u$ and $u'$ interchanged whenever
$m_{CD}(u')>0$, we obtain the reverse inequality. Hence
$m_{CD}(u)=m_{CD}(u')$; in particular, if $m_{CD}(u)=0$, then
$m_{CD}(u')=0$.

Since $C$ is connected as a component of $G_{\vis}$, the function $m_{CD}$ is
constant on $V(C)$.
\end{proof}

Next we strengthen Lemma~\ref{lem:other-visible-component-neighbour-count-constant} to $m_{CD}\leq 1$. When $C=D$, we have $m_{CC}=0$.

\begin{proposition}
\label{prop:at-most-one-neighbour-in-other-visible-component}
Let $C$ and $D$ be two connected components of $G_{\vis}$. Then for
every $x\in V(C)$,
\begin{equation*}
m_{CD}(x):=\card{S_1^{G_{\col}}(x)\cap V(D)}\le1.
\end{equation*}
If $C=D$, then $m_{CC}(x)=0$.
\end{proposition}

\begin{proof}
Choose $f\in E_2$ as in Lemma~\ref{lem:visible-generic-eigenfunction}. By
Lemma~\ref{lem:other-visible-component-neighbour-count-constant}, the number $m_{CD}(x)$ is constant on $V(C)$.  Choose $x_0\in V(C)$ such that $f(x_0)=\max_{x\in V(C)}f(x)$. If $m_{CD}(x_0)=0$, then we are done. Otherwise, we assume there exists $y_0\in S_1^{G_{\col}}(x_0)\cap V(D)$.

Since $x_0y_0$ is collapsed, $f(y_0)=f(x_0)$. Let $y\in S_1^{G_{\vis}}(y_0)$. Applying Proposition~\ref{prop:residual-visible-exactness} to $(y_0;y,x_0)$, we obtain a unique exact square completion $x$ for $(y_0;y,x_0)$.
\eqref{eq:midpoint} gives
\begin{equation}\label{eq:at-most-one-local-difference}
f(y)-f(y_0)=f(x)-f(x_0)\neq 0.
\end{equation}
Hence $x_0x$ is visible, so $x\in V(C)$. Combining with $f(x)\leq f(x_0)$, we have $f(x)<f(x_0)$. By \eqref{eq:at-most-one-local-difference}, we get $f(y)<f(y_0)$. By Lemma~\ref{lem:hypercube-unique-strict-local-maximum}, $y_0$ is a maximum of $f\vert_{V(D)}$.

Let $y^\prime\in S_1^{G_{\col}}(x_0)\cap V(D)$. Then $y^\prime$ is a maximum of $f\vert_{V(D)}$. By Theorem~\ref{thm:G_vis}, the visible component
$D$ is isomorphic to $H_r$. A direct computation shows $f|_{V(D)}\in E_2(D)$. Since $f(x)\neq f(y)$ for any edge $xy$ of $D$, Lemma~\ref{lem:hypercube-unique-strict-local-maximum} implies that $f|_{V(D)}$ has only one maximum $y_0$. Thus $y^\prime=y_0$. Then $m_{CD}(x_0)\leq 1$. Finally, Lemma~\ref{lem:other-visible-component-neighbour-count-constant} gives, for every
$x\in V(C)$,
\begin{equation*}
m_{CD}(x)
=
m_{CD}(x_0)
\le1.
\end{equation*}

When $C=D$, $x_0$ is the unique maximum of $f\vert_{V(C)}$. Hence $y_0$ couldn't exist. Thus $m_{CC}(x)=m_{CC}(x_0)=0$.
\end{proof}

From the above proposition, we know collapsed edges only join two distinct connected components of $G_{\vis}$. Any two vertices in the same connected component of $G_{\vis}$ couldn't be connected by a collapsed edge.

Let $C,D$ be two distinct connected component of $G_{\vis}$. Assume $xy$ is a collapsed edge where $x\in C,y\in D$. We define a map $\sigma_{CD}: V(C)\to V(D)$. By Proposition~\ref{prop:at-most-one-neighbour-in-other-visible-component}, $\card{S_1^{G_{\col}}(u)\cap V(D)}=1$ for $u\in V(C)$. Let $\sigma_{CD}(u)$ denote the unique neighbor of $u$ in $D$.
\begin{proposition}
\label{prop:Sigma_CDIsomorphism}
    $\sigma_{CD}$ is a graph isomorphism.
\end{proposition}
\begin{proof}
We first show that $\sigma_{CD}$ is a homomorphism. Fix $u\in V(C)$. Let $v\in S_1^{G_{\vis}}(u)$ and $u^\prime=\sigma_{CD}(u)$. By Lemma~\ref{lem:TransportMap}, $T_{uu^\prime}(v)\in S_1^{G_{\vis}}(u^\prime)$. Since $T_{uu^\prime}(v)$ is the exact square completion for $(u;v,u^\prime)$, $v\sim T_{uu^\prime}(v)$. Hence we get
\begin{equation}
\label{eq:sigma-equal-T}
    \sigma_{CD}(v)=T_{uu^\prime}(v).
\end{equation}
Hence $\sigma_{CD}(uv)=\sigma_{CD}(u)\sigma_{CD}(v)$ and $\sigma_{CD}$ is a homomorphism.

     Next, we show $\sigma_{CD}$ is a bijection. If $\sigma_{CD}(V(C))\neq V(D)$, choose an edge $u^\prime v^\prime$ such that $u^\prime\in \sigma_{CD}(V(C)),\ v^\prime\notin \sigma_{CD}(V(C))$. Let $u=\sigma_{CD}^{-1}(u^\prime)$. Since $T_{uu^\prime}$ is a bijection, by \eqref{eq:sigma-equal-T}, we have
     \begin{equation*}
         \sigma_{CD}\bigl(S_1^{G_{\vis}}(u)\bigr)=T_{uu^\prime}\bigl(S_1^{G_{\vis}}(u)\bigr)=S_1^{G_{\vis}}(u^\prime).
     \end{equation*}
     Hence $v^\prime\in \sigma_{CD}(V(C))$. It is a contradiction. Finally, since $\card{V(C)}=\card{V(D)}$, we also have $\sigma_{CD}$ is injective. Hence it is a bijection.
 \end{proof}

 \begin{theorem}\label{thm:visible-quotient-embeds}
Let $G$ be a connected regular Lichnerowicz-sharp graph and $G^\prime$ be the quotient graph whose vertices are the visible components of $G$, with two distinct visible components adjacent when they are joined by an edge of
$G$.  Then there exists a graph embedding
\begin{equation*}
s:G^\prime\hookrightarrow G
\end{equation*}
such that $s(C)\sim_{G} s(D)$ if and only if $C\sim_{G^\prime} D$.
\end{theorem}

\begin{proof}
Choose $f\in E_2$ as in Lemma~\ref{lem:visible-generic-eigenfunction}. By Lemma~\ref{lem:hypercube-unique-strict-local-maximum}, there exists a unique local strict maximum of $f|_{C}$ for any connected component $C$ of $G_{\vis}$. Let $s(C)$ be the unique local strict maximum. Since distinct visible components are disjoint, $s$ is injective.

It remains to prove that $s$ is a homomorphism.  Let $C,D$ be distinct visible
components with $C\sim_{G^\prime}D$ and $x=s(C)$.  Since $|S_1^{G}(x)\cap V(D)|=1$, let $y$ be the unique neighbor of $x$ in $D$. For any $z\in S_1^{G_{\vis}}(x)$, by \eqref{eq:midpoint}
\begin{equation*}
    f(z)-f(x)=f(T_{xy}(z))-f(y)<0.
\end{equation*}
Since $T_{xy}$ is a bijection, we have $f(y)$ is also a local strict maximum of $f\vert_{V(D)}$. Thus $s(D)=y$ and $s(C)\sim_G s(D)$. Then $s$ is an embedding.

Finally, if $C\ne D$ and $s(C)\sim_G s(D)$, then an edge
of $G$ joins $C$ and $D$, so $C\sim_{G^\prime}D$ by the definition of $G^\prime$.
\end{proof}

\begin{proof}[Proof of Theorem~\ref{thm:HypercubeBundle}]
By Theorem~\ref{thm:G_vis}, every connected component of $G_{\vis}$ is isomorphic to $H_r$ for $r\geq 1$. If $G$ doesn't have any collapsed edge, then $G=G_{\vis}$ and $G^\prime$ is a single vertex. Now assume $G$ has at least one collapsed edge.

Since $G$ is connected, the quotient graph $G^\prime$ is connected.  Let $C,D\in V(G^\prime)$ be distinct visible components. If $C\sim_{G^\prime}D$,  Proposition~\ref{prop:at-most-one-neighbour-in-other-visible-component} implies that every vertex in $C$ has exactly one neighbour in $D$. Besides, there exists no collapsed edge join two vertices lying the same connected component of $G_{\vis}$. Thus $G^\prime$ is a $(d-r)$-regular graph.

For every oriented edge $CD$ of $G^\prime$, define
\[
 \sigma_{CD}:V(C)\to V(D)
\]
by setting $\sigma_{CD}(u)$ to be the unique vertex of $D$ adjacent to $u$.  By
Proposition~\ref{prop:Sigma_CDIsomorphism}, $\sigma_{CD}$ is a graph isomorphism and by uniqueness, $\sigma_{DC}=\sigma_{CD}^{-1}$.

Choose, for every $C\in V(G^\prime)$, a graph isomorphism $ \iota_C:H_r\to C$. For an oriented edge $CD$ of $G^\prime$, set
\[
 \tau_{CD}:=\iota_{D}^{-1}\circ \sigma_{CD}\circ \iota_C\in {\rm Aut}(H_r).
\]
Then $\tau_{DC}=\tau_{CD}^{-1}$.  Consider the $H_r$-bundle
$G^\prime\square_\tau H_r$ in the sense of Definition~\ref{def:GraphBundle}.  Define
\[
 \Psi:V(G^\prime\square_\tau H_r)\to V(G),
 \quad
 \Psi(C,a)=\iota_C(a).
\]
This map is bijective.  If $a\sim b$ in $H_r$, then $\Psi(C,a)\sim\Psi(C,b)$ is a visible
edge in $C$.  If $C\sim_{G^\prime} D$, then
\[
 \Psi(D,\tau_{CD}a)=\sigma_{CD}(\iota_C(a)),
\]
so $\Psi(C,a)$ is adjacent to $\Psi(D,\tau_{CD}a)$ by a collapsed edge.  Conversely,
every edge of $G$ is either a visible edge inside a visible component, or a collapsed edge
between two distinct visible components.  Therefore $\Psi$ is a graph isomorphism.  Hence
$G$ is an $H_r$-bundle over $G^\prime$.

It remains to prove the curvature and spectral assertions for $G^\prime$.  Let $h:V(G^\prime)\to\mathbb R$
and lift it to $G$ by setting $ \widetilde h(x):=h(C_x)$, where $C_x$ is the visible component containing $x$. Then
\begin{equation*}
    L_G\widetilde{h}(x)=\sum_{y\sim x}\bigl(\widetilde{h}(y)-\widetilde{h}(x)\bigr)=\sum_{xy\in E_{\col}}\bigl(\widetilde{h}(y)-\widetilde{h}(x)\bigr)=\sum_{D\sim_{G^\prime}C_x}\bigl(h(D)-h(C_x)\bigr)=L_{G^\prime}h(C_x).
\end{equation*}
By \eqref{eq:Gamma} and \eqref{eq:Gamma2}, we have
\[
 \Gamma_G(\widetilde h)(x)=\Gamma_{G^\prime}(h)(C_x),\quad \Gamma_{2,G}(\widetilde h)(x)=\Gamma_{2,G^\prime}(h)(C_x).
\]
Since $G$ satisfies $\CD(2,\infty)$,
\[
 \Gamma_{2,G^\prime}(h)(C_x)=\Gamma_{2,G}(\widetilde h)(x)
 \ge 2\Gamma_G(\widetilde h)(x)=2\Gamma_{G^\prime}(h)(C_x).
\]
Thus $G^\prime$ satisfies $\CD(2,\infty)$.

Theorem~\ref{thm:LichnerowiczEstimate} gives $\lambda_1(G^\prime)\ge2$.  Suppose, for contradiction, that
$\lambda_1(G^\prime)=2$.  Let $0\ne h\in E_2(G^\prime)$ be nonconstant.  Then the lifted function
$\widetilde h$ satisfies $L_G\widetilde h=-2\widetilde h$, so $\widetilde h\in E_2(G)$.
Since $G^\prime$ is connected and $h$ is nonconstant, there is an edge $C\sim_{G^\prime} D$ with
$h(C)\ne h(D)$.  For any $u\in V(C)$, the edge
$u\sigma_{CD}(u)$ is collapsed in $G$, but
\[
 \widetilde h(u)=h(C)\ne h(D)=\widetilde h(\sigma_{CD}(u)),
\]
contradicting the definition of a collapsed edge.  Hence $\lambda_1(G^\prime)>2$.

Finally, Theorem~\ref{thm:visible-quotient-embeds} gives a graph embedding
$s:G^\prime\hookrightarrow G$ with $s(C)\in V(C)$ for every visible component $C$.  Equivalently,
under the bundle identification above, $s$ is a section of the projection
$G^\prime\square_\tau H_r\to G^\prime$.  This completes the proof.
\end{proof}
\begin{example}
\label{ex:noncartesian-k4-h2-cd2-lambda2}
Consider the base graph $K_4$ and fiber graph $H_2$. Denote vertex sets of $K_4$ and $H_2$ by
\[
V(K_4)=\{0,1,2,3\},\quad
V(H_2)=\mathbb F_2^2.
\]
Let $s:\mathbb F_2^2\to \mathbb F_2^2:
s(a,b)=(b,a)$. Define $\sigma$ by
\[
\sigma_{0i}=\sigma_{i0}=\operatorname{id},\ (i=1,2,3)\text{ and } \sigma_{ij}=\sigma_{ji}=s,\ (1\le i<j\le 3).
\]

Let $X=K_4 \square_\sigma H_2$. $X$ is a $5$-regular graph. We verify \(\CD(2,\infty)\). There are two local types. If \(s(u)=u\), i.e.
\(u\in\{00,11\}\), the curvature matrix is
\[
A_\infty=
\begin{pmatrix}
5&-1&-1&0&0\\
-1&5&-1&0&0\\
-1&-1&5&0&0\\
0&0&0&2&0\\
0&0&0&0&2
\end{pmatrix}.
\]
Its eigenvalues are $2,2,3,6,6$.
If \(s(u)\ne u\), i.e. \(u\in\{01,10\}\), then the local curvature matrix is
\[
A_\infty=2I_5.
\]
Thus every vertex satisfies $A_\infty\ge 2I_5$. Therefore $X$ is $\CD(2,\infty)$.

The Laplacian characteristic polynomial is
\[
\chi_L(t)
=
t(t-2)(t-4)^7(t-6)^3(t-8)^4.
\]
Therefore $\lambda_1(X)=2$ and $X$ is Lichnerowicz-sharp. $X_{\vis}$ is exactly four copies of $H_2$.

Now we prove that $X=K_4\square_\sigma H_2$ does not admit a nontrivial
Cartesian product decomposition. 
\begin{proof}
     Suppose, to the contrary, that
$X\cong A\square B$, where $A$ and $B$ are connected graphs with at least two
vertices.  Since $X$ is connected and $5$-regular, both $A$ and $B$ are
connected regular graphs, say of degrees $d_A$ and $d_B$, with
\[
d_A+d_B=5.
\]
Moreover, $|V(A)||V(B)|=|V(X)|=16$.  After interchanging the two factors if
necessary, there are only two cases.

First suppose that $|V(A)|=2$ and $|V(B)|=8$.  Then $A\cong K_2$.  If
$\mu_1,\ldots,\mu_8$ are the Laplacian eigenvalues of $B$, counted with
multiplicity, then the Laplacian eigenvalues of $K_2\square B$ are
\[
\mu_1,\ldots,\mu_8,\quad
\mu_1+2,\ldots,\mu_8+2.
\]
Since $B$ is connected, $0$ is a simple eigenvalue of $B$.  The characteristic
polynomial of $X$ above shows that $2$ is a simple eigenvalue of $X$.  The copy
of $0+2$ already accounts for this eigenvalue, so $2$ is not an eigenvalue of
$B$.  Consequently, the multiplicity $7$ of the eigenvalue $4$ in $X$ must
come entirely from the eigenvalue $4$ of $B$.  Thus the spectrum of $B$ would
be $\{0,4^{(7)}\}$, and $K_2\square B$ would have the eigenvalue $6$ with
multiplicity $7$.  This contradicts the fact that $6$ has multiplicity $3$ in
$X$.

It remains to consider $|V(A)|=|V(B)|=4$.  A connected regular graph on four
vertices has degree $2$ or $3$: in the former case it is $C_4$, and in the
latter it is $K_4$.  Since $d_A+d_B=5$, after interchanging the factors we
must have $A\cong C_4$ and $B\cong K_4$.  However, the eigenvalue $2$ has
multiplicity $2$ in $C_4\square K_4$, whereas it has multiplicity $1$ in
$X$.  This is again a contradiction.  Therefore $X$ has no nontrivial
Cartesian product decomposition.
\end{proof}
\end{example}

\section{Lichnerowicz-sharp hypercube bundles}
By Theorem~\ref{thm:HypercubeBundle}, every connected \(d\)-regular
Lichnerowicz-sharp graph admits the structure of a hypercube bundle over a
graph satisfying \(\CD(2,\infty)\) and \(\lambda_1>2\). Let $G^\prime$ be a $\CD(2,\infty)$ graph with $\lambda_1(G^\prime)>2$. Fix a fiber graph $H_r$ with $r\geq 1$. A natural question is which conditions on $\sigma=\{\sigma_{xy}\in\operatorname{Aut}(H_r):xy\in E(G^\prime)\}$ ensure that $G^\prime\square_\sigma H_r$ is a Lichnerowicz-sharp graph. In this section we
study the converse direction.

Let $G^\prime\square_\sigma F$ be any graph bundle. Let $\gamma=y_0\cdots y_n$ be any walk of $G^\prime$. Define $\sigma_{\gamma}$ as $\sigma_{y_{n-1}y_n}\circ \cdots\circ \sigma_{y_0y_1}$. For a base point $x_0\in V(G^\prime)$, let
 \begin{equation}
 \label{equ:PathGroup}
     \Gamma_{x_0}:=\{\sigma_\gamma: \text{$\gamma$ is a closed walk based at $x_0$}\},
 \end{equation}
and $G^\prime_{\mathrm{hor}}$ denote the horizontal subgraph of $G^\prime\square_\sigma F$ defined by
   \begin{equation}\label{eq:horizontal-cover-def-complete}
        G^\prime_{\mathrm{hor}}:=(V(G^\prime\square_\sigma F),\{(x,u)\sim(y,\sigma_{xy}(u))\text{ for $x\sim_{G^\prime} y$}\}),
    \end{equation}

We first consider, in a general setting, when $G^\prime\square_\sigma H_r$ has a $2$-eigenvector. Recall from \eqref{equ:DirectSum}, we have the orthogonal decomposition of $\mathbb{R}^{V(H_r)}$
 \begin{equation*}
 \mathbb{R}^{V(H_r)}=\bigoplus_{k=0}^r{W_k}(H_r),
 \end{equation*}
where $W_k(H_r)=\{\chi_S:S\subseteq \{1,\cdots,r\},\card{S}=k\}$ and $\chi_S(u)=(-1)^{\sum_{i\in S}u_i}$.
 \begin{proposition}
 \label{prop:EigenspaceOfBundle}
    Let $G^\prime$ be a connected graph with $\lambda_1(G^\prime)>2$ and $X=G^\prime\square_\sigma H_r$. Take a base point $x_0\in V(G^\prime)$. Let
      \begin{equation*}
         W_1(H_r)^{\Gamma_{x_0}}=\{f\in W_1(H_r): \text{$f\circ \sigma_\gamma=f$ for any $\sigma_\gamma\in \Gamma_{x_0}$}\}.
     \end{equation*}
    Then the restriction map $F\mapsto F|_{\{x_0\}\times H_r}$ induces an isomorphism
    \begin{equation*}
        E_2(X)\cong W_1(H_r)^{\Gamma_{x_0}}.
    \end{equation*}
    Moreover, for any $F\in E_2(X)$ and any edge $xy$ of $G^\prime$, $F(y,\sigma_{xy}(u))=F(x,u)$.
 \end{proposition}
 \begin{proof}
      Let
\begin{equation*}
    L_{\mathrm{hor}}f(x,u)=\sum_{y\sim x}\bigl(f(y,\sigma_{xy}(u))-f(x,u)\bigr),\quad L_{\mathrm{fib}}f(x,u)=\sum_{v\sim u}\bigl(f(x,v)-f(x,u)\bigr).
\end{equation*}
Then $L_X$ can be written as $L_X=L_{\mathrm{hor}}+L_{\mathrm{fib}}$. For any function $F$ on $X$, define
\begin{equation*}
    F_x(u):=F(x,u)\in \mathbb{R}^{V(H_r)}.
\end{equation*}

For $\tau\in \operatorname{Aut}(H_r)$, by \eqref{equ:HypercubeAutomorphism}, we can choose $a\in \mathbb{F}_2^r$ and a permutation $\pi\in S_r$ with
\begin{equation*}
        \tau(x)=a+\bigl(x_{\pi^{-1}(1)},\cdots,x_{\pi^{-1}(r)}\bigr).
\end{equation*}
So
\begin{equation*}
    \chi_S\circ \tau=(-1)^{\sum_{i\in S}a_i}\chi_{\pi^{-1}(S)}\in W_k(H_r).
\end{equation*}
Hence $L_{\mathrm{hor}}(W_k(H_r))\subseteq W_k(H_r)$. For $F\in E_2(X)$, decompose it as
\[
 F=\sum_{k=0}^r F^{(k)},
 \quad F^{(k)}_x\in W_k(H_r).
\]
Then
\begin{equation*}
    L_XF=\sum_{k=0}^r\bigl(L_{\mathrm{hor}}F^{(k)}+L_{\mathrm{fib}}F^{(k)}\bigr)=\sum_{k=0}^rL_{\mathrm{hor}}F^{(k)}-\sum_{k=0}^r2kF^{(k)}.
\end{equation*}
The equation $L_XF=-2F$ gives
\begin{equation*}
    L_{\mathrm{hor}}F^{(k)}=2(k-1)F^{(k)}.
\end{equation*}
We have for any $f\in \mathbb{R}^{V(X)}$,
\begin{equation}
\label{equ:LhorEnergy}
     \langle f,L_{\mathrm{hor}}f\rangle
 =-\frac12\sum_{xy\in E(G^{\prime})}\sum_{u\in V(H_r)}
 \bigl(f(y,\sigma_{xy}(u))-f(x,u)\bigr)^2\le 0.
\end{equation}
Thus $F^{(k)}=0$ for $k\ge2$.  For $k=0$, $F^{(0)}$ is constant on every fiber $\{x\}\times H_r$ and satisfies
$L_{G^{\prime}}F^{(0)}=-2F^{(0)}$.  Since $\lambda_1(G^{\prime})>2$, we conclude $F^{(0)}=0$.
Therefore $ F=F^{(1)}$ and $ L_{\mathrm{hor}}F=0$. \eqref{equ:LhorEnergy} gives
\begin{equation}
\label{equ:LhorEqual0}
    F(y,\sigma_{xy}(u))=F(x,u)
\end{equation}
for every edge $xy$ of $G^{\prime}$ and every $u\in H_r$.

Let \(f=F_{x_0}\in W_1(H_r)\). Iterating \eqref{equ:LhorEqual0} along a closed walk
\(\gamma\) based at \(x_0\) gives
\[
    f(\sigma_\gamma(u))=f(u),
    \quad
    \forall u\in V(H_r).
\]
Thus \(f\in W_1(H_r)^{\Gamma_{x_0}}\).

Conversely, if \(f\in W_1(H_r)^{\Gamma_{x_0}}\), choose for every \(x\in V(G^\prime)\) a path
\(p_x\) from \(x_0\) to \(x\), and define
\[
    F(x,u)=f(\sigma_{p_x}^{-1}(u)).
\]
The \(\Gamma_{x_0}\)-invariance of \(f\) makes this definition independent of the chosen path.
Moreover \(F_x\in W_1(H_r)\) for every \(x\), and \(F\) satisfies $F(y,\sigma_{xy}(u))=F(x,u)$.
Hence $L_{\mathrm{hor}}F=0$ and $L_{\mathrm{fib}}F=-2F$. So \(F\in E_2(X)\). This finishes the proof .
 \end{proof}
 \begin{theorem}
 \label{thm:Bundle2Eigenvector}
    Let $G^\prime$ be a connected graph with $\lambda_1(G^\prime)>2$. Let $X=G^\prime\square_\sigma H_r$. The following conditions are equivalent.
    \begin{enumerate}[label=\textup{(\roman*)}]
        \item $X$ admits a section;
        \item For any base point $x_0\in V(G^\prime)$, there exists $u_0\in V(H_r)$ such that for any $\sigma_\gamma\in \Gamma_{x_0}$, $\sigma_{\gamma}(u_0)=u_0$.
        \item There exists $F\in E_2(X)$ such that $F\bigl(x,u\bigr)\neq F\bigl(x,v\bigr)$ for $x\in V(G^\prime)$ and $u\sim_{H_r} v$.
    \end{enumerate}
 \end{theorem}
 \begin{proof}
    Suppose that $s$ is a section. Since $\pi_1\circ s=\operatorname{id}$, $s(x)$ has the following form
\[
 s(x)=(x,u_x), \quad u_x\in V(H_r).
\]
Since $s$ is a graph homomorphism, for every edge $xy$ of $G^\prime$, $(x,u_x)\sim_{X} (y,u_y)$. This gives $u_y=\sigma_{xy}(u_x)$. Let $\gamma=x_0\cdots x_n$ with $x_n=x_0$ be any closed walk based at $x_0$. By induction, $u_{i+1}=\sigma_{x_ix_{i+1}}(u_i)=\sigma_{x_0\cdots x_{i+1}}(u_0)$. Since $s(x_0)=s(x_n)$, we have $\sigma_{\gamma}(u_0)=u_0$. Thus \textup{(i)} implies \textup{(ii)}.

Conversely, assume \textup{(ii)}.  Let $u_0\in V(H_r)$ be a vertex fixed by $\sigma_{\gamma}$ for every closed walk $\gamma$ based at $x_0$.   For each vertex $x\in V(G^\prime)$, choose a path $p_x$ from $x_0$ to $x$ and define $ u_x:=\sigma_{p_x}(u_0)$. This is independent of the chosen path. Let $\widetilde{p}_x$ be another path join $x_0$ and $x$. Then
\begin{equation*}
    u_0=\sigma_{p_x*\widetilde{p}_x^{-1}}(u_0)=\sigma_{\widetilde{p}_x}^{-1}\bigl(\sigma_{p_x}(u_0)\bigr).
\end{equation*}
Thus $\sigma_{p_x}(u_0)=\sigma_{\widetilde{p}_x}(u_0)$. For any edge $xy$,
\begin{equation}
\label{equ:uy}
    u_y=\sigma_{p_y}(u_0)=\sigma_{p_x*xy}(u_0)=\sigma_{xy}(u_x).
\end{equation}
Thus $ s(x):=(x,u_x)$ defines a section.  Hence \textup{(ii)} implies \textup{(i)}.

Next we prove that \textup{(ii)} implies \textup{(iii)}.  Let $u_x=\sigma_{p_x}(u_0)$ where $p_x$ is any path from $x_0$ to $x$. Define
\[
 Q(x,u):=d(u,u_x)-\frac{r}{2}.
\]
For $(y,\sigma_{xy}(u))\sim_{X} (x,u)$, by \eqref{equ:uy} and isomorphism preserving distance, we have
\begin{equation*}
    Q(y,\sigma_{xy}(u))=d(\sigma_{xy}(u),\sigma_{xy}(u_x))-\frac{r}{2}=d(u,u_x)-\frac{r}{2}.
\end{equation*}
On the fibre $\{x\}\times H_r$, the function $Q(x,\cdot)$ belongs to \(W_1(H_r)\), and hence
$L_{\mathrm{fib}}Q=-2Q$. The equality above gives $L_{\mathrm{hor}}Q=0$, so $L_XQ=-2Q$.
If $d(u,u_x)=d(v,u_x)$, then $u\not\sim v$. Thus if $uv$ is an edge, we have
$|Q(x,u)-Q(x,v)|\geq 1$. This proves \textup{(iii)}.

It remains to prove that \textup{(iii)} implies \textup{(ii)}.

Fix the base point $x_0$.  Write $ F_{x_0}(u)=\sum_{i=1}^r a_i(-1)^{u_i}$.
Since $F_{x_0}(u)\neq F_{x_0}(v)$ for $u\sim_{H_r} v$, $a_i\ne0$ for all $i$. Let $u_0$ be a maximum of $F_{x_0}(u)$. For any closed walk $\gamma$ based at $x_0$, Proposition~\ref{prop:EigenspaceOfBundle} gives
\[
 F_{x_0}(\sigma_\gamma(u))=F_{x_0}(u),
 \quad\text{for all }u\in V(H_r).
\]
 By Lemma~\ref{lem:hypercube-unique-strict-local-maximum}, the maximum is unique. We conclude $\sigma_{\gamma}(u_0)=u_0$. Then $u_0$ is fixed by $\sigma_{\gamma}$ for any closed walk based at $x_0$. Thus \textup{(iii)} implies \textup{(ii)}.
 \end{proof}
\begin{remark}
    \label{rem:NotNecessaryCondition}
\begin{enumerate}[label=\textup{(\roman*)}]
   \item For \textup{(i)}$\Leftrightarrow$\textup{(ii)}, we do not need the condition $\lambda_1(G^\prime)>2$. But to make sure \textup{(iii)}$\Leftrightarrow$\textup{(ii)}, this condition is necessary.

   \item \textup{(iii)} and Proposition~\ref{prop:EigenspaceOfBundle} imply visible edges of $G^\prime\square_\sigma H_r$ are $(x,u)\sim (x,v)$ and its collapsed edges are $(x,u)\sim (y,\sigma_{xy}(u))$ for $x\sim_{G^\prime} y$ and $u\sim_{H_r}v$.
\end{enumerate}
\end{remark}

Proposition~\ref{prop:EigenspaceOfBundle} fully characterizes the $2$-eigenspace of a hypercube bundle over the base graph with $\lambda_1>2$ and Theorem~\ref{thm:Bundle2Eigenvector} ensures the existence of $2$-eigenvector if the bundle admits a section. However, it is more subtle to determine when the bundle
satisfies $\CD(2,\infty)$. In the remaining part of this section, we study
$\CD(2,\infty)$ bundles over some particular base graphs.
\subsection{Direct sum decomposition of curvature matrix}

We consider the curvature matrix of $G^\prime\square_\sigma F$.
The curvature matrix can be canonically decomposed into the direct sum of the
horizontal and fiber components. For a vertex
$P=(x,u)\in V(G^\prime\square_\sigma F)$, write
\begin{equation*}
    \mathcal{H}_P=\{(y,\sigma_{xy}(u)):y\sim_{G^\prime}x\},\quad \mathcal{F}_P=\{(x,v):v\sim_F u\}.
\end{equation*}
\begin{proposition}[Curvature matrix direct-sum decomposition]
\label{prop:DirectSumofCurvatureMatrix}
Let $F$ be a graph, let $X=G^\prime\square_\sigma F$, and let $G^\prime_{\mathrm{hor}}$ be the horizontal subgraph of $X$. Assume that for all triangles $xyzx$ of $G^\prime$ and $u\in V(F)$, $u\not\sim_F \sigma_{xyzx}(u)$.
For every vertex $P=(x,u)\in V(X)$, the curvature matrix $A_{\infty}^X(P)$ decomposes as
\begin{equation}\label{eq:horizontal-fibre-direct-sum-complete}
A_\infty^{X}(P)
=
A_\infty^{F}(u)\oplus A_\infty^{G^\prime_{\mathrm{hor}}}(P).
\end{equation}
\end{proposition}
\begin{proof}
Write
\[
    W_v=(x,v)\quad (v\sim_F u),\quad
    Q_y=(y,\sigma_{xy}(u))\quad (y\sim_{G^\prime}x).
\]

First consider $W_v$ and $Q_y$. We have $W_v\not\sim Q_y$. Consider the vertex
$R_{y,v}:=(y,\sigma_{xy}(v))$. This vertex belongs to $S_2^X(P)$ and is adjacent
to both $W_v$ and $Q_y$. We claim that it is the unique common neighbour $W_v, Q_y$ in
$S_2^X(P)$. Let $(z,w)\in S_2^X(P)$ be adjacent to both $W_v$ and $Q_y$. If
$(z,w)\neq R_{y,v}$, then both $(z,w)\sim W_v$ and $(z,w)\sim Q_y$ must be horizontal edges. Hence $x,y,z,x$
is a triangle in $G^\prime$ and
\[
   \sigma_{yz}\sigma_{xy}(u)=w=\sigma_{xz}(v).
\]
Thus $\sigma_{xyzx}(u)=v$, contradicting the assumption.
Moreover, if $Q_z\in \mathcal{H}_P\cap S_1^X(R_{y,v})$, then either $z=y$, or
the same closed-walk argument gives a contradiction. If
$W_w\in \mathcal{F}_P\cap S_1^X(R_{y,v})$, then $w=v$. Therefore
$S_1^X(P)\cap S_1^X(R_{y,v})=\{Q_y,W_v\}$.
Therefore
\begin{equation*}
    \varepsilon^X_{Q_y,W_v}=0,\quad \omega^X_{Q_y,W_v}=\frac12.
\end{equation*}

Next consider distinct vertices $W_v,W_w\in\mathcal{F}_P$. Their adjacency
relation in $X$ is exactly the adjacency relation of $v,w$ in $F$. A vertex in
$S_2^X(P)$ adjacent to both $W_v$ and $W_w$ cannot be joined horizontally to one
of them. Hence the common neighbours in
$S_2^X(P)$ are exactly the vertices $(x,a)$ with
$a\in S_2^F(u)$ adjacent to both $v$ and $w$. For such a vertex,
no horizontal neighbour of $P$ is adjacent to $(x,a)$. Hence
$d_P^{-,X}(x,a)=d_u^{-,F}(a)$ and
\begin{equation*}
    \varepsilon^X_{W_v,W_w}=\varepsilon^F_{v,w},\quad \omega^X_{W_v,W_w}=\omega^F_{v,w}.
\end{equation*}
For the diagonal entry, the preceding computation with the horizontal neighbours
gives
\begin{equation*}
    t_{W_v}^X=t_v^F,\quad \Omega_{W_v}^X=\Omega_v^F+\frac{\deg_{G^\prime_{\hor}}(P)}{2}.
\end{equation*}
Since $\deg_X(P)=\deg_F(u)+\deg_{G^\prime_{\hor}}(P)$ and
$\deg_X(W_v)=\deg_F(v)+\deg_{G^\prime_{\hor}}(P)$, \eqref{CKLP-off} and
\eqref{CKLP-diag} show that the $\mathcal{F}_P$-$\mathcal{F}_P$ block and the
$\mathcal{F}_P$-$\mathcal{H}_P$ block are $A_\infty^{F}(u)$ and the zero matrix,
respectively.

Finally consider the $\mathcal{H}_P$-$\mathcal{H}_P$ block. For $Q_y,Q_z\in\mathcal H_P$
with $y\ne z$, the vertex $Q_y$ is adjacent to $Q_z$ in $X$ if and only if $Q_y$
and $Q_z$ are adjacent in $G^\prime_{\mathrm{hor}}$. Let
$R=(w,a)\in S_2^X(P)$ be adjacent to both $Q_y$ and $Q_z$. If $w=y$, then
$a=\sigma_{xy}(b)$ for some $b\sim_F u$, and the adjacency $R\sim Q_z$ gives
$\sigma_{xzyx}(u)=b$, contradicting the assumption. Similarly, $w\neq z$.
Thus both adjacencies $R\sim Q_y$ and $R\sim Q_z$ are horizontal, so $R$ is also
a common neighbour of $Q_y$ and $Q_z$ in $G^\prime_{\mathrm{hor}}$. The converse

We also compare the values of $d_P^-$. Let
$R=(w,a)\in S_2^X(P)$ be a common neighbour of $Q_y$ and $Q_z$. If
$R$ were adjacent to a neighbour $W_b\in\mathcal{F}_P$, then
$a=\sigma_{xw}(b)$ for some $b\sim_F u$ and the horizontal adjacency
$R\sim Q_y$ would give $\sigma_{xywx}(u)=b$, a contradiction. If $R$ were
joined by a fiber edge to a neighbour $Q_w\in\mathcal{H}_P$, then
$a\sim_F\sigma_{xw}(u)$.  Writing $b=\sigma_{wx}(a)$ again gives
$b\sim_F u$ and $\sigma_{xywx}(u)=b$, a contradiction. Therefore
\begin{equation*}
     S_1^X(P)\cap S_1^X(R)=S_1^{G^\prime_{\mathrm{hor}}}(P)\cap S_1^{G^\prime_{\mathrm{hor}}}(R)
\end{equation*}
for every such common neighbour $R$, and
\begin{equation*}
    \varepsilon^X_{Q_y,Q_z}=\varepsilon^{G^\prime_{\mathrm{hor}}}_{Q_y,Q_z},
\quad
\omega^X_{Q_y,Q_z}=\omega^{G^\prime_{\mathrm{hor}}}_{Q_y,Q_z}.
\end{equation*}

Again by \eqref{CKLP-off}, we have
$(A_\infty^X(P))_{Q_y,Q_z}=\bigl(A_\infty^{G^\prime_{\mathrm{hor}}}(P)\bigr)_{Q_y,Q_z}$.
Now consider the diagonal entry $(A_\infty^X(P))_{Q_y,Q_y}$. We have
\begin{equation*}
    \begin{aligned}
        t_{Q_y}^X
        &=\sum_{Q\in S_1^X(P)\setminus\{Q_y\}}\varepsilon^X_{Q_y,Q}
          =\sum_{Q\in\mathcal{H}_P\setminus\{Q_y\}}\varepsilon^X_{Q_y,Q}
          =t_{Q_y}^{G^\prime_{\hor}},\\
        \Omega_{Q_y}^X
        &=\sum_{Q\in S_1^X(P)\setminus\{Q_y\}}\omega^X_{Q_y,Q}=\sum_{Q\in\mathcal{F}_P}\omega^X_{Q_y,Q}
          +\sum_{Q\in\mathcal{H}_P\setminus\{Q_y\}}\omega^X_{Q_y,Q}
          =\frac{\deg_F(u)}{2}+\Omega_{Q_y}^{G^\prime_{\hor}}.
    \end{aligned}
\end{equation*}

Since $\deg_X(P)=\deg_{G^\prime_{\hor}}(P)+\deg_F(u),
    \deg_X(Q_y)=\deg_{G^\prime_{\hor}}(Q_y)+\deg_F(u)$, \eqref{CKLP-diag} gives
\begin{equation*}
    \begin{aligned}
        (A_\infty^X(P))_{Q_y,Q_y}
&=3-\frac{\deg_X(P)+\deg_X(Q_y)}{2}
  +\frac{5}{2}t_{Q_y}^X+2\Omega_{Q_y}^X\\
&=3-\frac{\deg_{G^\prime_{\hor}}(P)+\deg_{G^\prime_{\hor}}(Q_y)}{2}
  +\frac{5}{2}t_{Q_y}^{G^\prime_{\hor}}
  +2\Omega_{Q_y}^{G^\prime_{\hor}}
  =(A_\infty^{G^\prime_{\mathrm{hor}}}(P))_{Q_y,Q_y}.
    \end{aligned}
\end{equation*}
\end{proof}
\begin{lemma}
    \label{lem:uNotAdjacentToSigmau}
    Let $G^\prime$ be a connected graph. Assume $G^\prime\square_\sigma H_r$ admits a section.  Then for any closed walk $\gamma$ of $G^\prime$ and every $u\in V(H_r)$, we have $u\not\sim \sigma_{\gamma}(u)$.
\end{lemma}
\begin{proof}
    Assume $G^\prime\square_\sigma H_r$ admits a section. Let $\gamma=x,\cdots,x_n=x$ be a closed walk in $G^\prime$. By Theorem~\ref{thm:Bundle2Eigenvector} and Remark \ref{rem:NotNecessaryCondition}, there exists a vertex $u_0\in V(H_r)$ such that $\sigma_\gamma(u_0)=u_0$ for every closed walk $\gamma$ based at $x\in V(G^\prime)$. Then
    \begin{equation*}
    d_{H_r}(\sigma_\gamma(v),u_0)= d_{H_r}(\sigma_\gamma(v),\sigma_\gamma(u_0))=d_{H_r}(v,u_0).
    \end{equation*}
    But for any $w\sim v$, $|d_{H_r}(u_0,w)-d_{H_r}(u_0,v)|=1$. Hence $ d_{H_r}(\sigma_\gamma(v),v)\ne 1$
for every $v\in V(H_r)$ and every closed walk $\gamma$ based at $x$ in $G^\prime$.
\end{proof}

From Theorem~\ref{thm:visible-quotient-embeds}, every connected regular Lichnerowicz-sharp graph admits a section. By Lemma~\ref{lem:uNotAdjacentToSigmau}, we can apply Proposition~\ref{prop:DirectSumofCurvatureMatrix} to the
regular Lichnerowicz-sharp graphs and get the following theorem.

\begin{theorem}
\label{thm:curvature-direct-sum-complete}
Let $G$ be a connected, $d$-regular Lichnerowicz-sharp graph.  For every vertex $P$,
the curvature matrix $A_{\infty}^G(P)$ decomposes into
\begin{equation}\label{eq:curvature-direct-sum-complete}
A_\infty^G(P)=A^{G_{\vis}}_{\infty}(P)\oplus A^{G_{\col}}_{\infty}(P).
\end{equation}
If one of the two neighbour sets is empty, the corresponding block is interpreted
as the empty block.
\end{theorem}

\begin{proof}
By Theorem~\ref{thm:HypercubeBundle}, \(G\) is isomorphic to a hypercube bundle
\(G^\prime\square_\sigma H_r\) and admits a section. Under this
identification, the fiber edges are precisely the visible edges, while the
edges of $G^\prime_{\hor}$ are precisely the collapsed edges. Hence for a vertex \(P=(x,u)\in V(G)\), the fiber block
\(A_\infty^{H_r}(u)\) in Proposition~\ref{prop:DirectSumofCurvatureMatrix}
is \(A_\infty^{G_{\vis}}(P)\), and the horizontal block
\(A_\infty^{G^\prime_{\mathrm{hor}}}(P)\) is \(A_\infty^{G_{\col}}(P)\). The
decomposition \eqref{eq:curvature-direct-sum-complete} follows.
\end{proof}

\subsection{Cartesian product}
When we take $\sigma_{xy}=\operatorname{id}_{H_r}$ for every edge $xy$ of $G^\prime$, we get the Cartesian product $G^\prime\square H_r$ of $G^\prime$ and $H_r$, that is the trivial bundle. The canonical embedding $G^\prime\to G^\prime\square H_r, x\to (x,0)$ gives a section of $G^\prime\square H_r$. Then by Lemma~\ref{lem:uNotAdjacentToSigmau} and Propositon~\ref{prop:DirectSumofCurvatureMatrix} gives the following the theorem
\begin{theorem}
  For any $(x,u)\in V(G^\prime\square H_r)$,  The curvature matrix $A_{\infty}^{G^\prime\square H_r}(x,u)$ has the following direct sum decomposition
  \begin{equation}
    \label{equ:DirectDecompositonOfCartesianProduct}
    A_{\infty}^{G^\prime\square H_r}(x,u)=A_{\infty}^{G^\prime}(x)\oplus A_{\infty}^{H_r}(u).
  \end{equation}
\end{theorem}

Since both $G^\prime$ and $H_r$ are $\CD(2,\infty)$, so is $G^\prime\square H_r$. By Theorem~\ref{thm:Bundle2Eigenvector}, $2$ is an eigenvalue of $G^\prime\square H_r$. By \eqref{equ:LichnerowiczEstimate}, we know $\lambda_1(G^\prime\square H_r)=2$. Thus Cartesian product is always a Lichnerowicz-sharp graph once $G^\prime$ is $\CD(2,\infty)$. 

\begin{figure}[htpb]
    \centering
    \begin{tikzpicture}[
  x={(1cm,0cm)},
  y={(0cm,1cm)},
  z={(0.62cm,0.38cm)},
  vertex/.style={circle, fill=black, inner sep=1.4pt},
  fiberedge/.style={red, line width=0.9pt, line cap=round},
  bundleedge/.style={black, line width=0.55pt, line cap=round}
]

\coordinate (a0) at (-1.35,  1.05, 0.00);
\coordinate (b0) at ( 1.35,  1.05, 0.00);
\coordinate (d0) at (-1.35, -1.05, 0.00);
\coordinate (c0) at ( 1.35, -1.05, 0.00);

\coordinate (fiberdir) at (0.78, 0.48, 0.90);
\coordinate (a1) at ($(a0)+(fiberdir)$);
\coordinate (b1) at ($(b0)+(fiberdir)$);
\coordinate (c1) at ($(c0)+(fiberdir)$);
\coordinate (d1) at ($(d0)+(fiberdir)$);

\foreach \u/\v in {a0/b0,b0/c0,c0/a0,a0/d0,b0/d0,c0/d0,a1/b1,b1/c1,c1/a1,a1/d1,b1/d1,c1/d1} {
  \draw[bundleedge] (\u) -- (\v);
}

\foreach \u/\v in {a0/a1,b0/b1,c0/c1,d0/d1} {
  \draw[fiberedge] (\u) -- (\v);
}

\foreach \v in {a0,a1,b0,b1,c0,c1,d0,d1} {
  \node[vertex] at (\v) {};
}

\end{tikzpicture}
    \caption{$K_2\square K_4$}
    \label{fig:K2-cartesian-K4}
\end{figure}
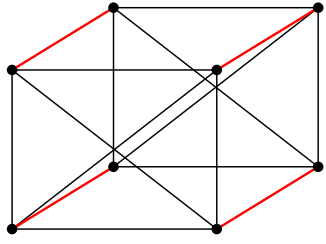

We can use the spectral information of $G^\prime\square_\sigma H_r$ to determine whether it is a trivial bundle. Before giving the theorem, we give the definition of isomorphism of two $F$-bundles over a base graph $G^\prime$.
\begin{definition}[{\cite[Defintation 3]{LL2024}}]
    Let $G^\prime\square_{\sigma_1}F$ and $G^\prime\square_{\sigma_2}F$ be two $F$-bundles over a base graph $G^\prime$. We say they are $F$-bundle isomorphic if there exists a graph isomorphism $\Psi: G^\prime\square_{\sigma_1}F\to G^\prime\square_{\sigma_2}F$ such that $\pi\circ \Psi=\pi$, where $\pi$ is the projection of $G^\prime\square_{\sigma_i}F$ to $G^\prime$ for $i=1,2$.
\end{definition}
We first give a splitting criterion for hypercube bundles over base graph with \(\lambda_1>2\).
\begin{theorem}
    \label{thm:BundleSplit}
     Let \(G^\prime\) be a connected graph with \(\lambda_1(G^\prime)>2\), and let \(X=G^\prime\square_{\sigma}H_r\). For \(1\leq m\leq r\), \(X\) is \(H_r\)-bundle isomorphic to \(B\square H_m\), where \(B\) is an \(H_{r-m}\)-bundle over \(G^\prime\), if and only if there exist \(m\) linearly independent \(2\)-eigenvectors \(f_1,\ldots,f_m\) such that $f_i(V(X))\subseteq \{-1,1\}$ for $1\leq i\leq m$.
\end{theorem}
\begin{proof}
Suppose first that
\[
    X\cong B\square H_m,
    \quad
    B=G^\prime\square_{\rho}H_{r-m},
\]
as an \(H_r\)-bundle over \(G^\prime\). Write a vertex of \(X\) as \((x,v,w)\), with
\(x\in V(G^\prime)\), \(v\in V(H_{r-m})\), and \(w\in V(H_m)\). For
\(1\leq i\leq m\), define $f_i(x,v,w)=\chi_i(w)$. Then \(L_Xf_i=-2f_i\). This
gives \(m\) linearly independent \(2\)-eigenvectors on \(X\) taking values in
\(\{-1,1\}\).

Conversely, assume that such \(f_1,\ldots,f_m\) exist. For every
\(x\in V(G^\prime)\), Proposition~\ref{prop:EigenspaceOfBundle} gives
\(f_i(x,\cdot)\in W_1(H_r)\). Since \(f_i(V(X))\subseteq\{-1,1\}\),
Lemma~\ref{lem:Binary2Eigenvector} implies that \(f_i(x,\cdot)\) is of the
form \(\pm\chi_{n_i}\) on every fibre. By Proposition~\ref{prop:EigenspaceOfBundle}
, the functions \(f_1(x,\cdot),\ldots,f_m(x,\cdot)\) are
linearly independent on $\{x\}\times H_r$ and thus, $n_i$ are pairwise distinct. 

For each \(x\in V(G^\prime)\), choose an automorphism
\(\theta_x\in\operatorname{Aut}(H_r)\) such that
\[
    \chi_{r-m+i}\circ\theta_x=f_i(x,\cdot),
    \quad 1\leq i\leq m.
\]
This is possible by the preceding paragraph and the description
\eqref{equ:HypercubeAutomorphism} of hypercube automorphisms. Then
\[
    \Theta:X\to G^\prime\square_{\tau}H_r,\quad
    \Theta(x,u)=(x,\theta_x(u)),
\]
is an \(H_r\)-bundle isomorphism, where
$
    \tau_{xy}=\theta_y\sigma_{xy}\theta_x^{-1}
$.
For every edge \(xy\) of \(G^\prime\), Proposition~\ref{prop:EigenspaceOfBundle}
gives $ f_i(y,\sigma_{xy}(u))=f_i(x,u)$.
Hence, for \(1\leq i\leq m\),
\[
    \chi_{r-m+i}\circ\tau_{xy}
    =
    \chi_{r-m+i}.
\]
Again using \eqref{equ:HypercubeAutomorphism}, we conclude that
\(\tau_{xy}\) fixes the last \(m\) coordinates and acts only on
the first \(r-m\) coordinates. Hence there exists
\(\rho_{xy}\in\operatorname{Aut}(H_{r-m})\) such that
\[
    \tau_{xy}(v,w)=(\rho_{xy}(v),w),
    \quad
    v\in V(H_{r-m}),\ w\in V(H_m).
\]
Therefore
\[
    G^\prime\square_{\tau}H_r
    \cong
    (G^\prime\square_{\rho}H_{r-m})\square H_m.
\]
Taking \(B=G^\prime\square_{\rho}H_{r-m}\), we get
\[
    X\cong B\square H_m
\]
as an \(H_r\)-bundle over \(G^\prime\).
\end{proof}

\begin{theorem}
\label{thm:TrivialBundle}
   Let $G^\prime$ be a connected graph with $\lambda_1(G^\prime)>2$ and $X=G^\prime\square_\sigma H_r$. The following conditions are equivalent:

    \begin{enumerate}[label=\textup{(\roman*)}]
        \item $X$ is an $H_r$-bundle isomorphic to $G^\prime\square H_r$;
        \item there exists $f\in E_2(X)$ such that $f\vert_{\{x\}\times H_r}$ is injective for every $x\in V(G^\prime)$;
        \item for any base point $x_0\in V(G^\prime)$, the group $\Gamma_{x_0}$ is trivial;
        \item $m_2(X)=r$.
    \end{enumerate}
\end{theorem}
\begin{proof}
Assume that $X$ is an $H_r$-bundle isomorphic to $G^\prime\square H_r$. Let $\Psi: X\to G^\prime\square H_r$ be the $H_r$-bundle isomorphism. Consider the $2$-eigenfunction $f$ on $G^\prime\square H_r$ defined by
\[
    f(x,u)=\sum_{i=1}^r 2^{i-1}(-1)^{u_i}.
\]
Then $f\circ \Psi$ is a $2$-eigenfunction of $X$ and its restriction to $\{x\}\times H_r$ is injective for every $x\in V(G^\prime)$. This proves \textup{(i)} implies \textup{(ii)}.

Next, we prove \textup{(ii)} implies \textup{(iii)}. Let \(f\in E_2(X)\) be such that
each restriction \(f|_{\{x\}\times H_r}\) is injective. In particular $h(u):=f(x_0,u)$ is injective on \(V(H_r)\) for any fixed base point $x_0$. By Proposition~\ref{prop:EigenspaceOfBundle}, for every $\sigma_{\gamma}\in \Gamma_{x_0}$, $h\circ\sigma_\gamma=h$. Since \(h\) is injective, this implies $\sigma_\gamma=\operatorname{id}_{H_r}$. Since \(x_0\) was arbitrary, \(\Gamma_{x_0}\) is trivial for every base point \(x_0\).

If \textup{(iii)} holds, then Proposition~\ref{prop:EigenspaceOfBundle} gives
\begin{equation*}
    m_2(X)=\dim W_1(H_r)^{\Gamma_{x_0}}=\dim W_1(H_r)=r.
\end{equation*}
This proves \textup{(iii)} implies \textup{(iv)}.

It remains to prove \textup{(iv)} implies \textup{(i)}. By Proposition~\ref{prop:EigenspaceOfBundle},
\[
    \dim W_1(H_r)^{\Gamma_{x_0}}
    =
    m_2(X)
    =
    r.
\]
Since \(\dim W_1(H_r)=r\), we have \(W_1(H_r)^{\Gamma_{x_0}}=W_1(H_r)\).
For \(1\leq i\leq r\), let \(F_i\in E_2(X)\) be the eigenfunction whose
restriction to \(\{x_0\}\times H_r\) is \(\chi_i\). Along any path \(p_x\) from \(x_0\) to \(x\), Proposition~\ref{prop:EigenspaceOfBundle}
also gives
\[
    F_i(x,u)=\chi_i(\sigma_{p_x}^{-1}(u)).
\]
Thus the functions \(F_i\) take values in \(\{-1,1\}\) on every fibre and remain
linearly independent. Applying
Theorem~\ref{thm:BundleSplit} with \(m=r\), we get
\[
    X\cong B\square H_r
\]
as an \(H_r\)-bundle over \(G^\prime\), where \(B\) is an \(H_0\)-bundle over
\(G^\prime\). Since \(H_0\) is the one-vertex graph, \(B\cong G^\prime\).
Thus \(X\) is \(H_r\)-bundle isomorphic to \(G^\prime\square H_r\). This proves
\textup{(iv)} implies \textup{(i)}.
\end{proof}
\begin{remark}
Let $G$ be a regular Lichnerowicz-sharp graph and $H_r$ be one connected component of $G_{\vis}$. Applying \textup{(iv)} in Theorem~\ref{thm:TrivialBundle} to $G$, we know if $G$ is not isomorphic to the Cartesian product of $G^\prime$ and $H_r$, then the multiplicity of the $2$-eigenvalue is strictly smaller than the degree $r$ of $G_{\vis}$.
\end{remark}
\subsection{Locally cubical graph}
We call a graph \(G\) locally cubical if for any two distinct neighbours
\(y,z\in S_1(x)\), the triple \((x;y,z)\) satisfies exact square incidence. When $G$ is locally cubical, it is automatically regular.
\begin{lemma}
    Let $G$ be a connected locally cubical graph. Then $G$ is regular.
\end{lemma}
\begin{proof}
Suppose, for contradiction, there exists an edge $xy$ such that $\deg x > \deg y$. For each $z\in S_1(x)\setminus\{y\}$, the triple
$(x;y,z)$ satisfies exact square incidence. Let $T(z)$ be its unique exact
square completion. Then $T(z)\sim y$, and $T(z)\neq x$, so
$T(z)\in S_1(y)\setminus\{x\}$. Moreover, $S_1(x)\cap S_1(T(z))=\{y,z\}$.
We claim the map
\[
    T:S_1(x)\setminus\{y\}\to S_1(y)\setminus\{x\}
\]
is injective. If $T(z_1)=T(z_2)$, then
$\{y,z_1\}=S_1(x)\cap S_1(T(z_1))
=S_1(x)\cap S_1(T(z_2))=\{y,z_2\}$, and so $z_1=z_2$. Therefore
$\deg(x)\leq \deg(y)$. This is a contradiction. Since $G$ is connected, $G$ is regular.
\end{proof}
\begin{remark}
    \label{rem:LocallyCubical}
    When connected graph $G$ is locally cubical, direct computation gives for any $x$, $\varepsilon_{ij}=0$ and $\omega_{ij}=\frac{1}{2}$. Thus $A_\infty(x)=2I_{\deg G}$.
\end{remark}

Next, we characterize locally cubical graphs through the curvature matrix.
\begin{theorem}
\label{thm:CharacterizationOfLocallyCubical}
    Let $G$ be a connected graph. If, for every $x\in V(G)$, $A_\infty(x)=2I_{\deg(x)}$, then $G$ is a regular locally cubical graph.
\end{theorem}

\begin{proof}
Fix a vertex \(x\in V(G)\), and write $S_1(x)=\{y_1,\dots,y_d\}$.

Since \(A_\infty(x)=2I_{\deg(x)}\), by \eqref{CKLP-off}, $ 1-2\varepsilon_{ij}-2\omega_{ij}=0$. If \(\varepsilon_{ij}=1\), then
\[
    1-2\varepsilon_{ij}-2\omega_{ij}
    =
    -1-2\omega_{ij}<0,
\]
which is impossible. Therefore $\varepsilon_{ij}=0$ and $\omega_{ij}=\frac12$.
Thus, for every pair of distinct neighbours \(y_i,y_j\in S_1(x)\), we have $y_i\not\sim y_j$. In particular, $G[S_1(x)]$ has no edge.

We next show that adjacent vertices have the same degree. \eqref{CKLP-diag} gives
\[
    (A_\infty(x))_{ii}
    =
    3-\frac{\deg(x)+\deg(y_i)}2+\frac52t_i+2\Omega_i.
\]
Since \(t_i=0\) and $\omega_{ij}=\frac{1}{2}$ for every $i\neq j$, we have $\Omega_i=\frac{d-1}{2}$.
Using \((A_\infty(x))_{ii}=2\), we obtain
\[
    2
    =
    3-\frac{d+\deg(y_i)}2+d-1
    =
    2+\frac{d-\deg(y_i)}2.
\]
Hence $\deg(y_i)=\deg(x)$. Because \(G\) is connected, it follows that \(G\) is regular. Write the degree of $G$ as $d$.

It remains to prove the exact square incidence. If \(d\leq 1\), then there are no two distinct
neighbours at any vertex, so exact square incidence automatically holds. Now assume \(d\geq 2\).

Fix again a vertex \(x\). Since \(G\) is \(d\)-regular and \(G[S_1(x)]\) has no edges, every
\(y_i\in S_1(x)\) has exactly \(d-1\) neighbours in \(S_2(x)\). Hence
\begin{equation}
\label{equ:SumOfBackDegree}
        \sum_{z\in S_2(x)} d^{-}_x(z)=d(d-1).
\end{equation}
On the other hand, since \(\omega_{ij}=\frac12\) for every \(i<j\), we have $\sum_{i<j}\omega_{ij}=\binom d2\frac12$.
Therefore
\[
    \sum_{i<j}\omega_{ij}
    =
    \sum_{z\in S_2(x)}
    \binom{d_{x}^{-}(z)}{2}\frac1{d_x^{-}(z)}
    =
    \sum_{z\in S_2(x)}
    \frac{d_x^{-}(z)-1}{2}.
\]
Hence $\sum_{z\in S_2(x)}(d_x^{-}(z)-1)=\binom d2$.
Subtracting this from \eqref{equ:SumOfBackDegree},
we get $\card{S_2(x)}=\binom{d}{2}$. Thus the average value of \(d_x^{-}(z)\) over \(z\in S_2(x)\) is
\begin{equation}
\label{equ:AverageOfBackDegree}
      \frac{\sum_{z\in S_2(x)}d_x^{-}(z)}{|S_2(x)|}
    =
    \frac{d(d-1)}{\binom d2}
    =
    2.
\end{equation}

We now prove that every \(z\in S_2(x)\) satisfies \(d_x^{-}(z)\geq 2\). Suppose not. Since
\(z\in S_2(x)\), we have \(d_x^{-}(z)\geq 1\). Thus assume there exists $z_0$ with $d_x^{-}(z_0)=1$. Write $S_1(x)\cap S_1(z_0)=\{y\}$. Since $A_\infty(y)=2I_{\deg(y)}$, \eqref{CKLP-off} at the vertex \(y\) gives $\omega^{\,y}_{x,z_0}=\frac12$.
In particular, there exists a vertex \(b\in S_2(y)\) such that $b\sim x,b\sim z_0$. Thus $b\in S_1(x)\cap S_1(z_0)$.
Since \(b\in S_2(y)\), we have \(b\neq y\). This contradicts $S_1(x)\cap S_1(z_0)=\{y\}$.
Therefore $d_x^{-}(z)\geq 2$, for every $z\in S_2(x)$. By \eqref{equ:AverageOfBackDegree}, we conclude $d_x^{-}(z)=2$ for every $z\in S_2(x)$.

Finally, take two distinct neighbours \(y_i,y_j\in S_1(x)\). Since $\omega_{ij}=\frac{1}{2}$ and every $z\in S_2(x)$ has $d_x^{-}(z)=2$,  there is exactly one such vertex \(w\in S_2(x)\) with $w\sim y_i,w\sim y_j$. Moreover,
\[
S_1(x)\cap S_1(w)=\{y_i,y_j\}.
\]
Thus every triple $(x;y_i,y_j)$ satisfies exact square incidence and \(G\) is locally cubical.
\end{proof}

The next theorem determines all $G^\prime\square_\sigma H_r$ satisfying $\CD(2,\infty)$ when $G^\prime$ is locally cubical.
\begin{theorem}
    Let $G^\prime$ be a connected locally cubical graph. $G^\prime\square_\sigma H_r$ satisfies $\CD(2,\infty)$ if and only if $G^\prime\square_\sigma H_r$ itself is locally cubical.
\end{theorem}
\begin{proof}
Let $X=G^\prime\square_\sigma H_r$. When $X$ is locally cubical, by the computation in Remark~\ref{rem:LocallyCubical}, $X$ is $\CD(2,\infty)$.

Now assume $X$ is $\CD(2,\infty)$. Let $d$ denote the degree of $G^\prime$. Since $G^\prime$ is triangle-free, by Proposition~\ref{prop:DirectSumofCurvatureMatrix}, for any $P=(x,u)\in V(X)$, we have
\begin{equation*}
    A_\infty^X(P)=A_\infty^{G^\prime_{\hor}}(P)\oplus A_\infty^{H_r}(u).
\end{equation*}
Since $X$ is $\CD(2,\infty)$, $A_\infty^{G^\prime_{\hor}}(P)\succeq 2I_{d}$.

Let $Q=(y,\sigma_{xy}(u)), W=(z,\sigma_{xz}(u))$ be two distinct neighbours of $P$ with $x\sim y, x\sim z$.  Let $w$ be the exact square completion for $(x;y,z)$ and
\begin{equation*}
    R_y:=(w,\sigma_{yw}\sigma_{xy}(u)),\ R_z:=(w,\sigma_{zw}\sigma_{xz}(u)).
\end{equation*}
Since $S_1^{G^\prime}(x)\cap S_1^{G^\prime}(w)=\{y,z\}$, the only possible vertices adjacent to both $P$ and $R_y$ in $G^\prime_{\hor}$ are $Q$ and $W$. If $R_y=R_z$, then $d_{P,G^\prime_{\hor}}^{-}(R_y)=2$. If $R_y\ne R_z$, then $d_{P,G^\prime_{\hor}}^{-}(R_y)=1$. We claim that \(R_y=R_z\) for all $Q,W$. Assume, for contradiction, there exists a pair $Q,W$ such that $R_y\ne R_z$. Then $d_{P,G^\prime_{\hor}}^{-}(R_y)=1$.

 Since $G^\prime$ is triangle-free, so is $G^\prime_{\hor}$.  Then $Q$ has,
besides \(P\), exactly $d-1$ neighbors in $S_2^{G^\prime_{\hor}}(P)$. Thus
\[
\sum_{\substack{\xi\in S_2^{G^\prime_{\hor}}(P)\\ \xi\sim Q}}
\frac1{d_{P,G^\prime_{\hor}}^{-}(\xi)}
\ge
1+\frac{d-2}{2}
=
\frac d2.
\]
By \eqref{CKLP-diag}, we obtain
\[
\begin{aligned}
(A_\infty^{G^\prime_{\hor}}(P))_{QQ}
&=
3-d+2(d-1)-
2\sum_{\substack{\xi\in S_2^{G^\prime_{\hor}}(P)\\ \xi\sim Q}}
\frac1{d_{P,G^\prime_{\hor}}^{-}(\xi)}\leq 1.
\end{aligned}
\]
Since $X$ is $\CD(2,\infty)$, we have $(A_\infty^{G^\prime_{\hor}}(P))_{QQ}\geq 2$. It is a contradiction. Therefore $R_y=R_z$. Then $R_y$ is the exact square completion for $(P;Q,W)$. Thus the horizontal graph \(G^\prime_{\hor}\) is locally cubical at \(P\), and
hence \(A_\infty^{G^\prime_{\hor}}(P)=2I_d\). Since \(A_\infty^{H_r}(u)=2I_r\),
\eqref{eq:horizontal-fibre-direct-sum-complete} gives \(A_\infty^X(P)=2I_{d+r}\). By
Theorem~\ref{thm:CharacterizationOfLocallyCubical}, \(X\) is locally cubical.
\end{proof}

\begin{example}
\label{ex:H-2n-plus-5-quotient} Consider $H_{2n+5}$ for $n\geq 1$. We write a vertex as
\[
x=(x_1,x_2,\ldots,x_{2n},u),
\quad
u=(u_1,\ldots,u_5)\in\mathbb F_2^5.
\]
Define a graph automorphism $\gamma:\mathbb F_2^{2n+5}\to \mathbb F_2^{2n+5}$ by
\[
\gamma(x_1,x_2,x_3,x_4,\ldots,x_{2n-1},x_{2n},u)
=
(x_2,x_1,x_4,x_3,\ldots,x_{2n},x_{2n-1},u+\mathbf 1),
\]
where $\mathbf 1=(1,1,1,1,1)\in\mathbb F_2^5$. Define the quotient graph $G_n:=H_{2n+5}/\langle\gamma\rangle $. We denote the orbit of \(x\in V(H_{2n+5})\) by \([x]\).

Since \(\gamma^2=\mathrm{id}\), the group \(\langle\gamma\rangle\) has order \(2\).
Moreover, \(\gamma\) has no fixed point. Hence
\[
|V(G_n)|
=
\frac{|V(H_{2n+5})|}{2}
=
2^{2n+4}.
\]

Since $d_{H_{2n+5}}(x,\gamma x)\ge 5$,  the quotient map $\pi:H_{2n+5}\to G_n$ restricts to an isomorphism from \(B_2(x)\) onto \(B_2([x])\). Since the hypercube is locally cubical, the same is true for \(G_n\). Thus $G_n$ is $\CD(2,\infty)$. Since both $G_n$ and $H_{2n+5}$ are $2n+5$-regular graph, but $\card{V(G_n)}\neq  \card{V(H_{2n+5})}$, the two graphs are not isomorphic.

We now compute the first nonzero Laplacian eigenvalue. Any $2$-eigenfunction of $G_n$ can be lifted to a $2$-eigenfunction  of $H_{2n+5}$. Conversely, every $2$-eigenfunction $f$ with $f\circ\gamma=f$ gives a $2$-eigenfunction of $G_n$. Thus
\begin{equation*}
    E_2(G_n)\cong \{f:f\in E_2(H_{2n+5}),\  f\circ \gamma=f\}.
\end{equation*}
Let $f=\sum_{i=1}^{2n+5}a_i\chi_i$. Then
\begin{equation*}
    f\circ \gamma=\sum_{k=1}^n\bigl(a_{2k-1}\chi_{2k}+a_{2k}\chi_{2k-1}\bigr)-\sum_{i=1}^5a_{2n+i}\chi_{2n+i}.
\end{equation*}
Therefore
\begin{equation}
\label{equ:EigenspaceOfG_n}
    \{f:f\in E_2(H_{2n+5}),\  f\circ \gamma=f\}=\operatorname{span}
\left\{
\chi_1+\chi_2,\,
\chi_3+\chi_4,\,
\ldots,\,
\chi_{2n-1}+\chi_{2n}
\right\}.
\end{equation}
Thus $m_2(G_n)=n\geq 1$. Hence $\lambda_1(G_n)\leq 2$. Then \eqref{equ:LichnerowiczEstimate} gives $\lambda_1(G_n)=2$. So $G_n$ is Lichnerowicz-sharp.

For $[x]\in V(G_n)$, by \eqref{equ:EigenspaceOfG_n}, the collapsed edges of $G_n$ are $[x]\sim[x+e_i]$ for $2n+1\leq i\leq 2n+5$ and the visible edges of $G_n$ are $[x]\sim [x+e_i]$ for $1\leq i\leq 2n$. Hence $(G_n)_{\vis}$ is the disjoint union of $2^{4}$ copies of $H_{2n}$ and $(G_n)^\prime$ is the quotient graph $\mathbb{F}_2^5/\langle u\sim u+\mathbf{1} \rangle$. $\mathbb{F}_2^5/\langle u\sim u+\mathbf{1} \rangle$ is also locally cubical. 

At last, we prove that $G_n$ does not admit a nontrivial Cartesian product
decomposition. 
\begin{proof}
    Suppose, to the contrary, that
\[
G_n\cong A\square B,
\]
where $A$ and $B$ are connected graphs with at least two vertices.  Since
$G_n$ is regular, both $A$ and $B$ are regular. We view $G_n$ as a trivla $B$-bundle over $A$. By Propostion~\ref{prop:DirectSumofCurvatureMatrix},
\begin{equation*}
    A^{G_n}_\infty(a,b)=A^{A}_\infty(a)\oplus A^B_\infty(b).
\end{equation*}
We get $A^{A}_\infty(a)=2I_{\deg(A)}, A^B_\infty(b)=2I_{\deg(B)}$. By Theorem~\ref{thm:CharacterizationOfLocallyCubical}, both $A$ and $B$ are locally cubical.

The Lichnerowicz estimate therefore gives
$\lambda_1(A),\lambda_1(B)\geq2$, and the product formula for the Laplacian
gives
\[
E_2(G_n)
=\bigl(E_2(A)\otimes\mathbf 1_B\bigr)
 \oplus
 \bigl(\mathbf 1_A\otimes E_2(B)\bigr).
\]
It follows that an edge $(a,b)\sim (a^\prime,b)$ of $A\square B$ is visible if
and only if the corresponding edge $a\sim a^\prime$ of $A$ is visible, and similarly for
$B$.  Consequently, if $A'$ and $B'$ denote the quotient graphs of $A$ and
$B$ by their visible components, then
\begin{equation}
\label{eq:visible-quotient-product-Gn}
(G_n)'\cong A'\square B'.
\end{equation}
Here if $\dim(E_2(A))=0$, we set $A^\prime=A$. It is similar for $B$.

We claim that there exists no non-trivial Cartesian product for $(G_n)'=\mathbb F_2^5/\langle u\sim u+\mathbf1\rangle$.  Indeed, its Laplacian eigenfunctions lift to those
characters $\chi_S$ of $H_5$ which satisfy
$\chi_S(u+\mathbf1)=\chi_S(u)$, or equivalently $|S|$ is even.  Hence
\[
\operatorname{Spec}\bigl(-L_{(G_n)'}\bigr)
=\{0^{(1)},4^{(10)},8^{(5)}\}.
\]
If $(G_n)'\cong C\square D$ were a nontrivial Cartesian product, then, since
$|(G_n)'|=16$ and $(G_n)'$ is connected and $5$-regular, the orders of the
two connected regular factors would be either $2$ and $8$, or $4$ and $4$.
In the first case the two-vertex factor is $K_2$, so the product has the
Laplacian eigenvalue $2$.  In the second case the factor degrees must be $2$
and $3$, and hence the factors are $C_4$ and $K_4$; this product again has
the eigenvalue $2$.  Both cases contradict the displayed spectrum, proving
the claim.

By \eqref{eq:visible-quotient-product-Gn}, one of $A'$ and $B'$ must therefore
be a single vertex; assume that $A'$ is a single vertex.  Thus $A_{\vis}$ is
connected.  Since $A$ has at least two vertices, $E_2(A)\neq0$ and hence
$\lambda_1(A)=2$.  Theorem~\ref{thm:G_vis} now shows that the unique visible
component of $A$ is a hypercube $H_s$ for some $s\geq1$.  Since collapsed
edges can only join distinct visible components, there are no further edges
in $A$, and therefore $A\cong H_s$.  In particular, $G_n$ has a Cartesian
factor $K_2$.

It remains to rule this out.  If $G_n\cong K_2\square C$, the function which
is $1$ on one $C$-layer and $-1$ on the other is a $2$-eigenfunction of
$G_n$ taking values in $\{-1,1\}$.  Its lift to $H_{2n+5}$ has the same
properties and is invariant under $\gamma$.  By
Lemma~\ref{lem:Binary2Eigenvector}, this lift must be $\pm\chi_i$ for some
$i$.  However, $\gamma$ interchanges $\chi_{2k-1}$ and $\chi_{2k}$ for
$1\leq k\leq n$, while
\[
\chi_i\circ\gamma=-\chi_i,
\qquad 2n+1\leq i\leq2n+5.
\]
Thus no $\pm\chi_i$ is invariant under $\gamma$, a contradiction.  We
conclude that $G_n$ admits no nontrivial Cartesian product decomposition.
\end{proof}
\end{example}

\subsection{Complete graph}
In this subsection, we study some hypercube bundles over complete graphs $K_n$. Recall from Theorem~\ref{thm:HypercubeBundle} that every connected regular Lichnerowicz-sharp graph admits a section.

We first consider the case of \(K_3\). The case is rigid: every Lichnerowicz-sharp hypercube bundle over \(K_3\) is trivial.

\begin{theorem}
\label{thm:K3-twist-CD2-product}
Let $X=K_3 \square_\sigma H_r$. Then $X$ is Lichnerowicz-sharp if and only if it is isomorphic to $K_3\square H_r$.
\end{theorem}
\begin{proof}
When $X\cong K_3\square H_r$, it is a Lichnerowicz-sharp graph. 

Suppose $X$ is Lichnerowicz-sharp. Write $V(K_3)=\{0,1,2\}$. By Theorem~\ref{thm:Bundle2Eigenvector}, we can choose $u_0\in H_r$ which is a common fixed point of $\Gamma_0$.
Set $\gamma=0120$. By assumption, $\sigma_\gamma(u_0)=u_0$. We show that $\sigma_\gamma=\operatorname{id}_{H_r}$.

By Theorem~\ref{thm:curvature-direct-sum-complete}, for every vertex $(x,u)\in V(X)$, 
\begin{equation*}
    A_\infty^X(x,u)=A_\infty^{(K_3)_{\hor}}(x,u)\oplus A_\infty^{H_r}(u).
\end{equation*}
Since $X$ is Lichnerowicz-sharp, $A_\infty^{(K_3)_{\hor}}(x,u)\succeq 2I_2$.

Assume, for contradiction, that \(\sigma_\gamma\ne \operatorname{id}_{H_r}\). Since \(\sigma_\gamma\)
fixes \(u_0\), by \eqref{equ:HypercubeAutomorphism},
there exists a neighbour \(v_0\sim u_0\) such that $\sigma_\gamma(v_0)\ne v_0$. Put
\begin{equation*}
    P=(0,v_0),\quad Q=(1,\sigma_{01}(v_0)),\quad W=(2,\sigma_{02}(v_0)).
\end{equation*}
Since $\sigma_\gamma(v_0)\ne v_0$, $Q\not\sim W$ in $(K_3)_{\hor}$. Thus $t_Q^{(K_3)_{\hor}}=0$. Because there exists no common neighbour of $Q$ and $W$ in $S_2^{(K_3)_{\hor}}(P)$, $\omega_{Q,W}^{(K_3)_{\hor}}=0$. By \eqref{CKLP-diag}, we have
\[
(A_\infty^{(K_3)_{\hor}}(P))_{QQ}
=
3-2=1.
\]
Since $X$ is $\CD(2,\infty)$, we have $(A_\infty^{(K_3)_{\hor}}(P))_{QQ}\geq 2$. It is a contradiction. Therefore $\sigma_\gamma=\operatorname{id}_{H_r}$. Since every closed walk in \(K_3\) based at \(0\) is generated, up to
backtracking, by the triangle \(0120\) and its inverse, this implies that
\(\Gamma_0\) is trivial. Then by Theorem~\ref{thm:TrivialBundle}, we have $X\cong K_3\square H_r$.
\end{proof}
For \(H_1\)-bundles over \(K_n\), every Lichnerowicz-sharp example is trivial.
Indeed, a Lichnerowicz-sharp hypercube bundle admits a section, whereas the
unique nontrivial automorphism of \(H_1\) has no fixed point.

Since $\lambda_1(K_2)=2$, Theorem~\ref{thm:HypercubeBundle} implies that
$K_2$ cannot occur as the base graph of a connected regular Lichnerowicz-sharp graph. However, for $n\geq 4$ and $r\geq 2$, we can construct non-trivial $H_r$-bundles over $K_n$ which are Lichnerowicz-sharp.
\begin{example}
Fix \(r\geq 2\). Let \(V(K_n)=\{0,1,\ldots,n-1\}\), and let
\(\tau\in\operatorname{Aut}(H_r)\) be the automorphism interchanging the first two coordinates:
\[
    \tau(u_1,u_2,u_3,\ldots,u_r)=(u_2,u_1,u_3,\ldots,u_r).
\]
For every oriented edge $ij$, define $\sigma_{ij}=\tau$. Set \(X_{n,r}:=K_n\square_\sigma H_r\). Since
\(\tau^2=\operatorname{id}_{H_r}\), every $\sigma_\gamma\in \Gamma_0$ is either
\(\operatorname{id}_{H_r}\) or \(\tau\). The triangle \(0120\) gives
\[
    \sigma_{0120}=\sigma_{20}\sigma_{12}\sigma_{01}=\tau^3=\tau,
\]
so \(\Gamma_0=\{\operatorname{id}_{H_r},\tau\}\). Hence \(\Gamma_0\) is
nontrivial. Note that every \(u\in\mathbb F_2^r\) with \(u_1=u_2\) is fixed by every element of \(\Gamma_0\).

We first show that \(X_{n,r}\) is \(\CD(2,\infty)\). For every triangle
\(xyzx\) of \(K_n\), we have \(\sigma_{xyzx}=\tau\). Since
\(d_{H_r}(u,\tau(u))\) is either \(0\) or \(2\), the hypothesis of
Proposition~\ref{prop:DirectSumofCurvatureMatrix} holds. Thus
\[
    A_\infty^{X_{n,r}}(P)
    =
    A_\infty^{(K_n)_{\hor}}(P)\oplus A_\infty^{H_r}(u),
\]
where $P=(i,u)\in V(X_{n,r})$.

Now we check $A_\infty^{(K_n)_{\hor}}(P)$. For each fixed point \(u\) of
\(\tau\), the corresponding connected component of \((K_n)_{\hor}\) is \(K_n\). A direct
computation gives
\[
    A_\infty^{(K_n)_{\hor}}(P)=\frac{3n}{2}I_{n-1}-J_{n-1}\succeq 2I_{n-1},
\]
where \(J_{n-1}\) denotes the all-one matrix. For $u$ with $u\neq\tau(u)$, let $Q=(j,\tau(u)), W=(k,\tau(u))$ be two distinct neighbours of \(P=(i,u)\) in $(K_n)_{\hor}$. We have $W\not\sim Q$. The common neighbours of $Q$ and $W$ in $S_2^{(K_n)_{\hor}}(P)$ are $(l,u)$ with $l\ne i,j,k$. Every $(l,u)$ has $d_{P,(K_n)_{\hor}}^{-}(l,u)=n-2$. Thus $\omega_{Q,W}^{(K_n)_{\hor}}=\frac{n-3}{n-2}$. 
Hence
\[
    (A_\infty^{(K_n)_{\hor}}(P))_{QQ}=n-2,
    \quad
    (A_\infty^{(K_n)_{\hor}}(P))_{QW}=-\frac{n-4}{n-2}.
\]
The eigenvalues of this matrix are \(2\) and \(\frac{n(n-3)}{n-2}\). Since
\(n\geq 4\), \((K_n)_{\hor}\) is \(\CD(2,\infty)\). Therefore \(X_{n,r}\) is
\(\CD(2,\infty)\).

Since \(\Gamma_0\) has a common fixed point, Theorem~\ref{thm:Bundle2Eigenvector}
implies that \(X_{n,r}\) admits a \(2\)-eigenfunction. By \eqref{equ:LichnerowiczEstimate}, \(\lambda_1(X_{n,r})=2\).
Hence \(X_{n,r}\) is Lichnerowicz-sharp.

Finally, since \(\Gamma_0\) is nontrivial, Theorem~\ref{thm:TrivialBundle}
implies that \(X_{n,r}\) is not \(H_r\)-bundle isomorphic to
\(K_n\square H_r\). Thus \(X_{n,r}\) is a nontrivial Lichnerowicz-sharp
hypercube bundle over \(K_n\).
\begin{figure}[htpb]
    \centering
     \scalebox{0.65}{\begin{tikzpicture}[
  x={(1cm,0cm)},
  y={(0cm,1cm)},
  z={(0.56cm,0.34cm)},
  vertex/.style={circle, fill=black, inner sep=1.4pt},
  fiberedge/.style={red, line width=0.9pt, line cap=round},
  bundleedge/.style={black, line width=0.55pt,
    dash pattern=on 5pt off 3pt, line cap=round},
  solidbundleedge/.style={blue, line width=0.9pt, line cap=round}
]


\def\outer{0.50}
\def\inner{0.25}
\def\baseR{3.05}

\foreach \f/\ang in {0/90,1/162,2/234,3/306,4/18} {
  \foreach \p/\dx/\dy in {
    000/-\outer/-\outer,
    100/ \outer/-\outer,
    110/ \outer/ \outer,
    010/-\outer/ \outer,
    001/-\inner/-\inner,
    101/ \inner/-\inner,
    111/ \inner/ \inner,
    011/-\inner/ \inner
  } {
    \coordinate (v-\f-\p) at ({\baseR*cos(\ang)+\dx},{\baseR*sin(\ang)+\dy},0);
  }
}

\foreach \a/\b in {0/1,0/2,0/3,0/4,1/2,1/3,1/4,2/3,2/4,3/4} {
  \foreach \p/\tp in {000/000,100/010,010/100,110/110,001/001,101/011,011/101,111/111} {
    \draw[bundleedge] (v-\a-\p) -- (v-\b-\tp);
  }
}

\foreach \a/\b in {0/1,0/2,0/3,0/4,1/2,1/3,1/4,2/3,2/4,3/4} {
  \draw[solidbundleedge] (v-\a-000) -- (v-\b-000);
}

\foreach \f in {0,1,2,3,4} {
  \foreach \u/\v in {000/100,010/110,001/101,011/111,000/010,100/110,001/011,101/111,000/001,100/101,010/011,110/111} {
    \draw[fiberedge] (v-\f-\u) -- (v-\f-\v);
  }
}

\foreach \f in {0,1,2,3,4} {
  \foreach \p in {000,100,010,110,001,101,011,111} {
    \node[vertex] at (v-\f-\p) {};
  }
}

\end{tikzpicture}}
    \caption{$X_{5,3}$}
\end{figure}
\end{example}

\section{Application}
\subsection{Some rigidity theorems}
\label{subsection:Some rigidity theorems}
We now prove the three rigidity statements summarized in
Theorem~\ref{tm:Bundle-structure2}. Let \(G\) be a connected \(d\)-regular
Lichnerowicz-sharp graph. We consider successively the cases
\(m_2(G)\geq d-1\), \(m_2(G)=d-2\), and \(m_2(G)=d-3\).

\begin{proof}[Proof of \ref{item1: d-1} of Theorem~\ref{tm:Bundle-structure2}]
  Let $r$ denote the degree of $G_{\vis}$. By Proposition~\ref{prop:EigenspaceOfBundle}, we have $m_2(G)\leq r$. Since $m_2(G)\geq d-1$, we get $r\geq d-1$. If $r=d$, $G$ doesn't have any collapsed edge. Hence by Theorem~\ref{thm:HypercubeRigidity}, $G$ is a $d$-dimensional hypercube. 

  Now assume $r=d-1$. Then $G^\prime$ is a $1$-regular graph, that is, $G^\prime\cong K_2$. However by Theorem~\ref{thm:HypercubeBundle}, we have $\lambda_1(G^\prime)>2$, which is a contradiction. Thus $G$ is a $d$-dimensional hypercube.
\end{proof}

Now we consider the case where \(m_2(G)=d-2\). 

\begin{proof}[Proof of \ref{item2: d-2} of Theorem~\ref{tm:Bundle-structure2}]
    Let $r$ denote the degree of $G_{\vis}$. By Proposition~\ref{prop:EigenspaceOfBundle}, we have $m_2(G)\leq r$. Since $m_2(G)=d-2$, we get $r\geq d-2$. If $r=d$, then $G$ is a $d$-dimensional hypercube. Then \(m_2(G)=d\), which contradicts the assumption. If $r=d-1$, then $G^\prime$ is a $1$-regular graph, that is, $G^\prime\cong K_2$. But by Theorem~\ref{thm:HypercubeBundle}, we have $\lambda_1(G^\prime)>2$, which is a contradiction. Thus $r=d-2$. 
    
    Since $m_2(G)=d-2=r$, $G^\prime$ is a $2$-regular graph. Since $G^\prime$ is connected, we have $G^\prime\cong C_n$ for some $n\geq 3$. By Theorem~\ref{thm:HypercubeBundle}, we have $\lambda_1(G^\prime)>2$. Since $\lambda_1(C_n)=2-2\cos(2\pi/n)$, we have $n=3$. By Theorem~\ref{thm:TrivialBundle} \textup{(iv)}, $G$ is isomorphic to $H_{d-2}\square K_3$. 
\end{proof}

Before proving \ref{item3: d-3} of Theorem~\ref{tm:Bundle-structure2}, we need the following classification of connected $3$-regular graphs satisfying $\CD(2,\infty)$ and $\lambda_1>2$.
\begin{lemma}
    \label{lem:3RegularGraph}
    Let $G$ be a connected $3$-regular graph satisfying $\CD(2,\infty)$ and $\lambda_1(G)>2$. Then $G\cong K_4$ or $G\cong K_{3,3}$.
\end{lemma}
\begin{proof}
Since $G$ satisfies $\CD(2,\infty)$, it also satisfies $\CD(0,\infty)$. By the classification of connected $3$-regular graphs with nonnegative Bakry--\'Emery curvature
\cite[Theorem~1.1]{CKLLS}, $G$ is either a prism graph or a M\"obius ladder.
\begin{figure}[htpb]
    \centering
    \begin{tikzpicture}[
  vertex/.style={circle, fill=black, inner sep=1.7pt},
  edge/.style={black, line width=0.55pt, line cap=round},
  omitted/.style={black, line width=0.55pt, dash pattern=on 5pt off 6pt, line cap=round}
]

\def\ytop{1.0}
\def\ybot{0.0}
\def\xzero{0.0}
\def\xone{1.3}
\def\xtwo{2.6}
\def\xthree{5.8}
\def\xfour{7.1}

\foreach \x in {\xzero,\xone,\xtwo,\xthree,\xfour} {
  \draw[edge] (\x,\ybot) -- (\x,\ytop);
}
\draw[edge] (\xzero,\ytop) -- (\xone,\ytop) -- (\xtwo,\ytop);
\draw[edge] (\xzero,\ybot) -- (\xone,\ybot) -- (\xtwo,\ybot);
\draw[omitted] (\xtwo,\ytop) -- (4.0,\ytop);
\draw[omitted] (4.4,\ytop) -- (\xthree,\ytop);
\draw[omitted] (\xtwo,\ybot) -- (4.0,\ybot);
\draw[omitted] (4.4,\ybot) -- (\xthree,\ybot);
\draw[edge] (\xthree,\ytop) -- (\xfour,\ytop);
\draw[edge] (\xthree,\ybot) -- (\xfour,\ybot);
\draw[edge] (\xzero,\ytop) to[out=18,in=162] (\xfour,\ytop);
\draw[edge] (\xzero,\ybot) to[out=-18,in=-162] (\xfour,\ybot);

\foreach \x in {\xzero,\xone,\xtwo,\xthree,\xfour} {
  \node[vertex] at (\x,\ytop) {};
  \node[vertex] at (\x,\ybot) {};
}

\node at (3.55,-0.95) {The prism graph \(C_n\square K_2\)};

\begin{scope}[yshift=-3.2cm]
  \foreach \x in {\xzero,\xone,\xtwo,\xthree,\xfour} {
    \draw[edge] (\x,\ybot) -- (\x,\ytop);
  }
  \draw[edge] (\xzero,\ytop) -- (\xone,\ytop) -- (\xtwo,\ytop);
  \draw[edge] (\xzero,\ybot) -- (\xone,\ybot) -- (\xtwo,\ybot);
  \draw[omitted] (\xtwo,\ytop) -- (4.0,\ytop);
  \draw[omitted] (4.4,\ytop) -- (\xthree,\ytop);
  \draw[omitted] (\xtwo,\ybot) -- (4.0,\ybot);
  \draw[omitted] (4.4,\ybot) -- (\xthree,\ybot);
  \draw[edge] (\xthree,\ytop) -- (\xfour,\ytop);
  \draw[edge] (\xthree,\ybot) -- (\xfour,\ybot);
  \draw[edge] (\xzero,\ytop) -- (\xfour,\ybot);
  \draw[edge] (\xzero,\ybot) -- (\xfour,\ytop);

  \foreach \x in {\xzero,\xone,\xtwo,\xthree,\xfour} {
    \node[vertex] at (\x,\ytop) {};
    \node[vertex] at (\x,\ybot) {};
  }

  \node at (3.55,-0.95) {The M\"obius ladder \(M_n\)};
\end{scope}

\end{tikzpicture}
    \caption{The prism graph \(C_n\square K_2\) and the M\"obius ladder \(M_n\).}
    \label{fig:prism-mobius}
\end{figure}

If $G$ is a prism graph, then $G\cong C_n\square K_2$ for some $n\geq 3$. Then
\[
    \lambda_1(C_n\square K_2)
    =
    \min\left\{2,\ 2-2\cos\frac{2\pi}{n}\right\}
    \leq 2.
\]
This contradicts $\lambda_1(G)>2$. 

Therefore $G$ is a M\"obius ladder. Write it as $M_n=\operatorname{Cay}\bigl(\mathbb{Z}_{2n};\{\pm 1,n\}\bigr)$, with $n\geq 2$. The nonzero Laplacian eigenvalues of $M_n$ are
\[
    3-\left(2\cos\frac{\pi k}{n}+(-1)^k\right),
    \quad k=1,\ldots,2n-1.
\]
For $n\geq 4$, taking $k=2$ gives
\[
    \lambda_1(M_n)\leq 2-2\cos\frac{2\pi}{n}\leq 2,
\]
which contradicts $\lambda_1(G)>2$. Hence $n=2$ or $n=3$. Since $M_2\cong K_4$ and $M_3\cong K_{3,3}$,
we conclude that $G\cong K_4$ or $G\cong K_{3,3}$.
\end{proof}

\begin{proof}[Proof of \ref{item3: d-3} of Theorem~\ref{tm:Bundle-structure2}]
Let $r$ denote the degree of $G_{\vis}$. By Proposition~\ref{prop:EigenspaceOfBundle},
\[
    d-3=m_2(G)\leq r.
\]
Hence $r\geq d-3$.

If $r=d$, then $G$ has no collapsed edge. By Theorem~\ref{thm:HypercubeRigidity}, we have $G\cong H_d$, and hence the multiplicity of eigenvalue $2$ is $d$, a contradiction. If $r=d-1$, then $G^\prime$ is a connected $1$-regular graph, so $G^\prime\cong K_2$. This contradicts Theorem~\ref{thm:HypercubeBundle}, which gives $\lambda_1(G^\prime)>2$ whenever collapsed edges exist.

It remains to exclude the case $r=d-2$. In this case, \(G^\prime\) is a connected $2$-regular graph, so \(G^\prime\cong C_n\) for some \(n\geq 3\). By Theorem~\ref{thm:HypercubeBundle}, we have \(\lambda_1(G^\prime)>2\). Since $ \lambda_1(C_n)=2-2\cos\frac{2\pi}{n}$, we get \(n=3\), and hence \(G^\prime\cong K_3\). By Theorem~\ref{thm:K3-twist-CD2-product}, \(G\cong K_3\square H_{d-2}\). Therefore the multiplicity of eigenvalue \(2\) is \(d-2\), again a contradiction.

Thus \(r=d-3\). Then \(G^\prime\) is a connected $3$-regular graph. By Theorem~\ref{thm:HypercubeBundle}, \(G^\prime\) satisfies \(\CD(2,\infty)\) and \(\lambda_1(G^\prime)>2\). Applying Lemma~\ref{lem:3RegularGraph}, we have $ G^\prime\cong K_4$ or $ G^\prime\cong K_{3,3}$.
Moreover $m_2(G)=d-3=r$. By Theorem~\ref{thm:TrivialBundle} \textup{(iv)}, \(G\) is isomorphic to the Cartesian product of \(G^\prime\) and \(H_r\). Hence $ G\cong H_{d-3}\square K_4$ or  $G\cong H_{d-3}\square K_{3,3}$.
\end{proof}

\subsection{Irregular case}
Our main goal is to extend part~\ref{item1: d-1} of
Theorem~\ref{tm:Bundle-structure2} to graphs that are not
necessarily regular, and to obtain a corresponding rigidity theorem.

We begin by recalling the following extension result for \(2\)-eigenfunctions, due to \cite{LMP}.
\begin{lemma}[{\cite[Lemmas 3.5 and 3.6]{LMP}}]\label{lem:one-ball}
Let \(G\) satisfy \(\operatorname{CD}(2,\infty)\). Assume that \(\lambda_{\deg(x)}=2\) for some vertex $x$. If \(f:B_1(x)\to\mathbb{R}\) satisfies
\(
Lf(x)=-2 f(x),
\)
then there exists a unique \(2\)-eigenfunction \(\phi\in E_2\) such that \(\phi|_{B_1(x)}=f\).
\end{lemma}

The following lemma is essentially contained in \cite[Proof of Lemma 3.7]{LMP}.
\begin{lemma}\label{lem:distance}
Let $G$ satisfy $\CD(2,\infty)$. Assume that $\lambda_{\mathrm{deg}(o)}=2$ for some vertex $o$. Then the function 
\[
  h_o(\cdot)=d(o,\cdot)-\frac{\mathrm{deg}(o)}{2}
\]
belongs to $E_2(G)$, and $\Gamma h_o\equiv r/2$.
\end{lemma}

\begin{proof}
Let $r:=\mathrm{deg}(o)$. On $B_1(o)$ define $f(o)=-r/2$ and $f(y)=1-r/2$ for $y\sim o$. Then we check that
\[
  Lf(o)=\sum_{y\sim o}(f(y)-f(o))
   =r=-2 f(o).
\]
Then Lemma \ref{lem:one-ball} gives a unique $2$-eigenfunction $\varphi$ such that \(\varphi|_{B_1(o)}=f\). We prove $\varphi\equiv h_o$ by induction. By construction, we have $\varphi\vert_{B_1(o)}=h_o\vert_{B_1(o)}.$
We assume $\varphi\vert_{B_k(o)}=h_o\vert_{B_k(o)}.$ Then for $z\in S_{k+1}(o) $ and $x\in S_{k-1}(o)$ with $d(x,z)=2,$ by Lemma \ref{lem:LMP-midpoint}, we have
\begin{equation*}
\begin{aligned}
     \varphi(z)=&-\varphi(x)+  \frac{2}{d_x^{-}(z)}
  \sum_{y\in S_1(x)\cap S_1(z)}\varphi(y)\\
  =&-h_o(x)+  \frac{2}{d_x^{-}(z)}
  \sum_{y\in S_1(x)\cap S_1(z)}h_o(y)=h_o(z).
\end{aligned}
\end{equation*}
Thus \(\varphi=h_o\) on \(B_{k+1}(o)\). By induction, \(\varphi\equiv h_o\), and hence \(h_o\in E_2(G)\).
Finally, by
Lemma \ref{lem:LMP-sharp}, $\Gamma h_o$ is constant, and hence \(\Gamma (h_o)(x)=\Gamma(h_o)(o)=r/2\) for any $x$. 
\end{proof}
A key observation is that the spectral condition in Lemma \ref{lem:distance}
already forces regularity.

\begin{theorem}\label{thm:local-regularity}
Let \(G\) be a connected graph satisfying \(\CD(2,\infty)\). If
\(\lambda_{\deg(o)}=2\) for some vertex \(o\in V\), then \(G\) is regular.
\end{theorem}
\begin{proof}
Let \(r:=\deg(o)\). By Lemma \ref{lem:distance}, the function
\(h_o=d(o,\cdot)-r/2
\)
belongs to \(E_2\), and \(\Gamma(h_o)\equiv r/2\). For every vertex
\(v\in V\), we have
\begin{equation}\label{eq:Gamma_distance}
  r
  =
  2\Gamma(h_o)(v)
  =
  \#\{u\sim v: d(o,u)\neq d(o,v)\}
  \leq \deg(v).
\end{equation}
Thus \(r=\delta\) is the minimum degree of \(G\).

Suppose that there exists a neighbor \(y\sim o\) such that \(\deg(y)>r\).
By \eqref{eq:Gamma_distance}, the strict inequality \(\deg(y)>r\) implies that \(y\) has a
neighbor \(z\in S_1(o)\). Define \(f:B_1(o)\to\mathbb{R}\) by
\[
  f(o)=0,
  \quad
  f(y)=1,
  \quad
  f(z)=-1,
  \quad
  f(u)=0
  \quad \text{for } u\in S_1(o)\setminus\{y,z\}.
\]
Then \(Lf(o)=0=-2f(o)\). By Lemma \ref{lem:one-ball}, there exists
\(\phi\in E_2\) such that \(\phi|_{B_1(o)}=f\). By Lemma
\ref{lem:LMP-sharp}, \(\Gamma(\phi)\) is constant. However,
\(2\Gamma(\phi)(o)=2,\)
whereas the two edges \(y\sim o\) and \(y\sim z\) alone give
\[
  2\Gamma(\phi)(y)
  \geq
  \bigl(\phi(y)-\phi(o)\bigr)^2
  +
  \bigl(\phi(y)-\phi(z)\bigr)^2
  =
  1+4
  =
  5.
\]
This contradicts the constancy of \(\Gamma(\phi)\). Hence every neighbor of
\(o\) has degree \(r\).

Now let \(x\) be any neighbor of \(o\). Since \(\deg(x)=r\) and
\(\lambda_r=2\), the same argument applied with \(x\) in place of \(o\)
shows that every neighbor of \(x\) also has degree \(r\). By connectedness,
all vertices of \(G\) have degree \(r\). Therefore \(G\) is regular.
\end{proof}

Next, let us recall the following multiplicity upper bound.
\begin{lemma}[{\cite[Theorem~3.8]{LMP}}]\label{lem:mult}
If $G$ is a satisfies $\CD(2,\infty)$, then
\(
     m_2(G)\le \delta.
\)
where $\delta$ is the minimum degree of $G$.
\end{lemma}

Now we can prove the following theorem
\begin{theorem}[Optimal spectral rigidity]\label{thm:Delta-1Rigidity}
Let $G$ be a connected graph satisfying $\CD(2,\infty)$, and
let $\Delta=\max_x\deg(x)$ be the maximum degree of $G$.  If $\lambda_{\Delta-1}=2$, then $ G\cong H_\Delta$.
\end{theorem}
\begin{proof}
By Lichnerowicz estimate \eqref{eq:Lich}, $\lambda_1\ge 2$, so $m_2(G)\ge \Delta-1$.  By Lemma \ref{lem:mult},
\[
  \Delta-1\le m_2(G)\le \delta.
\]
If $\delta=\Delta-1$, then Theorem \ref{thm:local-regularity} implies that $G$ is \((\Delta-1)\)-regular, contradicting the definition of $\Delta$.
Hence, one must have $\delta=\Delta$, i.e. $G$ is
$\Delta$-regular. Part~\ref{item1: d-1} of
Theorem~\ref{tm:Bundle-structure2} applies.
\end{proof}
\begin{remark}
    The hypothesis in Theorem~\ref{thm:Delta-1Rigidity} cannot be
weakened from $\lambda_\delta=2$ to $\lambda_{\delta-1}=2$. For $\delta\geq 2$, the
graph $G:=H_{\delta-2}\square\left(K_4\setminus e\right)$ has minimum degree $\delta$. Here, we use the notation that $H_0$ is the one-vertex graph.  Recall that $K_4\setminus e$ satisfies $\CD(2,\infty)$ \cite[Example 3.2]{LMP}. Hence, $G$ satisfies $\CD(2,\infty)$. Noticing that
\[
        \Spec(L_{K_4\setminus e})=\{0,2,4,4\},
\]
we have $\lambda_{\delta-1}(G)=2$ and $G$ is not a hypercube. 
\begin{figure}[htpb]
    \centering
    \begin{tikzpicture}[
  vertex/.style={circle, fill=black, inner sep=1.6pt},
  edge/.style={black, line width=0.55pt, line cap=round}
]

\coordinate (t) at (0,1.25);
\coordinate (b) at (0,-1.25);
\coordinate (l) at (-1.35,0);
\coordinate (r) at (1.35,0);

\foreach \u/\v in {l/t,t/r,r/b,b/l,l/r} {
  \draw[edge] (\u) -- (\v);
}

\foreach \v in {t,b,l,r} {
  \node[vertex] at (\v) {};
}

\end{tikzpicture}
    \caption{$K_4\setminus e$}
\end{figure}
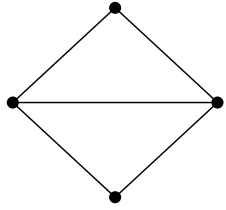
\end{remark}

\subsection{Splitting theorem}
Now we give the proof of Theorem~\ref{thm:K_2Factor}.

\begin{theorem}
    \label{thm:K_2Factor}
    Let $G=G^\prime\square_\sigma H_r$ be a hypercube bundle over a connected graph
    $G^\prime$ with $\lambda_1(G^\prime)>2$, where \(r\geq1\). Put
    \(M=m_2(G)\). If \(M>\frac r2\), then
    \[
        G\cong B\square H_{2M-r}
    \]
    for some connected graph \(B\), where \(B\) is an \(H_{2(r-M)}\)-bundle over
    \(G^\prime\). Moreover, \(m_2(B)=r-M\).
\end{theorem}
\begin{proof}[Proof of Theorem~\ref{thm:K_2Factor}]
By Proposition~\ref{prop:EigenspaceOfBundle}, we have $M\leq r$. Set \(N=2M-r\).  Then \(1\leq N\leq r\). We prove by induction that, for each
\(0\leq k\leq N\), there exists a connected graph \(B_k\) such that
\[
    G\cong B_k\square H_k,
\]
where \(B_k\) is an \(H_{r-k}\)-bundle over \(G^\prime\) and $m_2(B_k)=M-k$.

The case \(k=0\) holds with \(B_0=G\). Suppose this holds for some \(k<N\). Then $  M-k>\frac{r-k}{2}$. Proposition~\ref{prop:EigenspaceOfBundle} and
Lemma~\ref{lem:large-holonomy-invariant-coordinate} give a function $  F_k\in\ker(L_{B_k}+2I)$  with \(F_k(\{x\} \times V(H_{r-k}))\subseteq\{-1,1\}\) for some $x\in V(G^\prime)$. For each neighbor \(y\) of \(x\) in \(G^\prime\), since $F_k(y,\sigma_{xy}(u))=F_k(x,u)$, we also have $F_k(\{y\} \times V(H_{r-k}))\subseteq\{-1,1\}$. Since $G^\prime$ is connected, we conclude that \(F_k(V(B_k))\subseteq\{-1,1\}\). Theorem~\ref{thm:BundleSplit} with \(m=1\) then gives
\[
    B_k\cong B_{k+1}\square K_2,
\]
where \(B_{k+1}\) is a connected \(H_{r-k-1}\)-bundle over \(G^\prime\). By the
spectral formula for Cartesian products,
\[
    m_2(B_{k+1})=m_2(B_k)-1=M-k-1.
\]
This proves the induction.

Taking \(k=N\), we
obtain \(G\cong B_N\square H_N\). Since
\[
    N=2M-r,\quad r-N=2(r-M),
\]
the graph \(B_N\) is an \(H_{2(r-M)}\)-bundle over \(G^\prime\), and
\[
    m_2(B_N)=M-N=r-M.
\]
Taking \(B=B_N\) proves the theorem.
\end{proof}
Here the condition $M>\frac{r}{2}$ is sharp. Consider $G_n$ defined in Example~\ref{ex:H-2n-plus-5-quotient}. It satisfies $m_2(G_n)=n$ and $r=2n$, but it admits no nontrivial Cartesian product. Hence it can't split off any $H_1$ factor.

Having established the general splitting theorem for hypercube bundles, we now
apply it to regular Lichnerowicz-sharp graphs to prove
Theorem~\ref{thm:Lichnerowicz-sharpK_2Factor}.
\begin{proof}[Proof of Theorem~\ref{thm:Lichnerowicz-sharpK_2Factor}]
    Let $r$ denote the degree of $G_{\vis}$. If $r=d$, then by Theorem~\ref{thm:HypercubeRigidity}, $G\cong H_d\cong H_{2m-1}\square H_{d-2m+1}$. If $r<d$, then $m_2(G)\geq \frac{d}{2}>\frac{r}{2}$. Hence Theorem~\ref{thm:K_2Factor} gives 
    \begin{equation*}
    G\cong \widetilde{B}\square H_{2m_2(G)-r} \cong B\square H_{d-2m+1},
    \end{equation*}
where $B=\widetilde{B}\square H_{2m_2(G)+2m-r-d-1}$. 

Since \(G\) is \(d\)-regular and connected,
it follows that \(B\) is \((2m-1)\)-regular and connected. Since $B\square H_{d-2m+1}$ is $\CD(2,\infty)$, the direct decomposition \eqref{equ:DirectDecompositonOfCartesianProduct} for $B\square H_{d-2m+1}$ gives $B$ is also $\CD(2,\infty)$. Since $m_2(G)=m_2(B)+m_2(H_{d-2m+1})$, $m_2(B)\geq m-1$.
\end{proof}

In particular, when $m=1$ and $d\geq 2$, by the preceding theorem, we have $G\cong B\square H_{d-1}$ where $B$ is a connected $1$-regular graph. Hence $B\cong K_2$ and $G\cong H_d$. When $d=1$, we also have $G\cong K_2$. Thus we recover part~\ref{item1: d-1} of Theorem~\ref{tm:Bundle-structure2}.

\section{Acknowledgments}
This work was supported by the Scientific Research Innovation Capability Support Project for Young Faculty (Grant No. SRICSPYF-ZY2025160) and the National Natural Science Foundation of China (Grant No. 12431004).

\bibliography{rigid.bib}
\end{document}